\documentclass[11pt]{article}

\usepackage{latexsym,verbatim,ifthen,graphicx,mathrsfs}

\usepackage[english]{babel}
\usepackage[utf8]{inputenc} 	
\usepackage{amssymb}
\usepackage{amsmath}
\usepackage{amsthm}		
\usepackage{tikz-cd}
\usepackage{mathtools}  
\usepackage{microtype}	

\usetikzlibrary{matrix,arrows,decorations.pathmorphing}
\usepackage[all]{xy}				
\usepackage{hyperref}
\hypersetup{
	colorlinks=true,
	allcolors=blue,
}  		
\usepackage{anysize}
\usepackage{cancel}									
\marginsize{2.5cm}{2.5cm}{2cm}{2cm}
\usepackage{enumerate}

\usepackage{mathrsfs}									
\usepackage{fourier}			
\usepackage[Sonny]{fncychap}	
\usepackage{fancyhdr}

\newtheorem{theorem}{Theorem}[section]

\newtheorem{proposition}[theorem]{Proposition}
\newtheorem{corollary}[theorem]{Corollary}
\theoremstyle{definition}
\newtheorem{example}[theorem]{Example}

\theoremstyle{definition}
\newtheorem{remark}[theorem]{Remark}

\newtheorem{definition}[theorem]{Definition}

\newtheorem{thm}{Theorem}

\DeclareMathOperator{\Ker}{Ker}

\DeclareMathOperator{\sgn}{sgn}

\DeclareMathOperator{\id}{id}

\DeclareMathOperator{\mc}{MC}
\DeclareMathOperator{\MC}{MC_\bullet}
\DeclareMathOperator{\Map}{\mathsf{Map}}

\DeclareMathOperator{\sSet}{\mathsf{sSet}}

\newcommand{\Sull}{A_{\mathrm{PL}}}

\newcommand{\D}{\mathcal{\D}}
\newcommand{\Z}{{\mathbb{Z}}}
\newcommand{\Q}{{\mathbb{Q}}}

\newcommand{\C}{{\mathbb{C}}}

\newcommand{\GL}{G\textrm{-}\mathsf{s\widehat L_\infty}}
\newcommand{\GsSet}{G\textrm{-}\mathsf{sSet}}
\newcommand{\GCDGA}{G\textrm{-}\mathsf{CDGA}}

\definecolor{red}{rgb}{1,0.1,0.1}
\definecolor{blue}{rgb}{0.1,0.1,1}

\begin{document}

\title{On the homotopy fixed points of Maurer-Cartan spaces with finite group actions}

\author{José M. Moreno-Fernández, Felix Wierstra}

\date{}

\maketitle

\abstract{We develop the basic theory of Maurer-Cartan simplicial sets associated to (shifted complete) $L_\infty$ algebras
equipped with the action of a finite group.
Our main result  asserts that the inclusion of the fixed
points of this equivariant simplicial set into the homotopy fixed points is a homotopy equivalence of Kan complexes, provided the $L_\infty$ algebra is concentrated in non-negative degrees. 
As an application, and under certain connectivity assumptions, we provide rational algebraic
models of the fixed and homotopy fixed points of mapping spaces equipped with
the action of a finite group.}

\section{Introduction}

The Maurer-Cartan simplicial set $\MC(L)$ associated
to a complete $L_\infty$ algebra $L$, as well as its smaller analog, the Deligne-Hinich-Getzler $\infty$-groupoid $\gamma_\bullet(L)$, \cite{Hin97A,Get09} are of vital importance in
deformation theory, rational homotopy theory, mathematical physics, and other fields (see for example \cite{Bar93,Kon03,Prid10,Bui20,Rob20}).
Since these are such important objects, many of their properties are well understood.
Nevertheless, 
these objects have not yet been studied in the equivariant setting.

\smallskip

The goal of this paper is to explain how the Maurer-Cartan 
simplicial set extends to the $G$-equivariant setting as a functor $$\MC : G\textrm{-}\mathsf{s\widehat L_\infty} \to G\textrm{-}\mathsf{sSet}$$
from  the category of complete $sL_\infty$ algebras equipped with a $G$-action to simplicial sets with a $G$-action. 
We study the properties of this functor in the case when $G$ is finite.
In practice, we work with \emph{shifted} $L_\infty$ algebras, i.e. under the convention that all higher $L_\infty$ brackets are of degree $-1$.
More about this choice in Section \ref{MCSimplicial}.
Our motivation to study the extension of the Maurer-Cartan simplicial set to the $G$-equivariant setting arises mainly from the study of the
rational homotopy type of spaces
endowed with a $G$-action.
We will recall the necessary background on 
equivariant homotopy theory and explain how to equip the Maurer-Cartan simplicial set with the action of a group when the associated complete $sL_\infty$ algebra carries such an action.
If $L$ is an $sL_\infty$ algebra 
with a $G$-action, we say that $L$ is 
a $G$-$sL_\infty$ algebra.

As is well-known, 
for a $G$-space or simplicial set $X$, 
the classical fixed points space $X^G$ under a  given $G$-action is homotopically speaking poorly behaved.
This problem is overcome by forming one of the central constructions in $G$-equivariant homotopy theory: the space of \emph{homotopy fixed points}, $$X^{hG} := \Map_G\left(EG,X\right).$$
Here, $EG$ is a contractible space with a free $G$-action.
Collapsing $EG\xrightarrow{\simeq} *$ provides an inclusion $X^G \hookrightarrow X^{hG}$. 
The main result in this paper reads as follows and can be found as Theorem \ref{HomEq}.

\begin{thm}\label{Teo:Main}\em
		Let $G$ be a finite group and  $L=L_{\geq 1}$ a positively graded complete $G$-$sL_\infty$-algebra. 
		Then, the natural inclusion 
	\begin{equation*}
		\MC \left(L\right)^G \hookrightarrow \MC\left(L\right)^{hG}
	\end{equation*} is a homotopy equivalence of Kan complexes. 
\end{thm} 

For any complete $G$-$sL_\infty$ algebra $L$,
its underlying graded vector space of $G$-fixed points $L^G$ can be canonically endowed with a complete $sL_\infty$ algebra structure.
Furthermore, 
there is a natural isomorphism of Kan complexes $$\MC\left(L^G\right) \cong \MC\left(L\right)^G.$$
Therefore, 
Theorem \ref{Teo:Main} replaces the problem of computing the homotopy fixed points space $\MC(L)^{hG}$ by the much simpler 
problem of computing the space  $\MC\left(L^{G}\right)$.
See Remark \ref{Remark1}.

Theorem \ref{Teo:Main} above is a special case of the so-called \emph{generalized Sullivan conjecture} \cite{Car91}.
In particular, 
under the hypotheses of Theorem \ref{Teo:Main}
it implies (see Corollary \ref{Homotopy Groups}) that, at the canonical basepoint 
(namely, the $0$ Maurer-Cartan element), 
there is an isomorphism
$$\pi_*\left( \MC(L)^{hG}\right) \cong \big(\pi_*\MC(L)\big)^G.$$

A weaker form of Theorem \ref{Teo:Main}
can be deduced from Goyo's thesis \cite{Goy89}; 
our result extends and complements some of the main results 
there. 
Roughly speaking, for a finite group $G$,
Goyo 
studied the  inclusion of the fixed points into the homotopy fixed points for the Sullivan realization of a commutative differential graded algebra (CDGA, henceforth).
He required the CDGA's to be 1-connected  finite type Sullivan algebras.
Under those hypotheses, 
the CDGA's are in bijective correspondence with degree-wise nilpotent $sL_\infty$ algebras concentrated in positive degrees,
and the simplicial realization of Sullivan is isomorphic to the Maurer-Cartan simplicial set of the corresponding $L_\infty$ algebra. 
Our result 
only requires non-negatively graded $L_\infty$ algebras and extends from finite type to complete $L_\infty$ algebras,
thus providing a broader class of spaces where the result holds.  
Furthermore, our proof of Theorem \ref{Teo:Main} is completely different from Goyo's approach: 
It is simplicial, and does not rely on the Federer-Schultz spectral sequence
or on the passage to CDGA's, staying in the simplicial Lie world.
The proof of Theorem $\ref{Teo:Main}$ and its corollary is carried out in Section \ref{SullivanConjecture}.

\smallskip

In Section \ref{LieModels}, we 
apply our methods to produce 
rational algebraic models of 
equivariant mapping spaces.
That is, algebraic models of the space $\Map\left(X,\MC(L)\right)$ endowed with the conjugation action, and of the homotopy fixed points of this mapping space. 
Being more precise, let $X$ be a $G$-simplicial set, and $L=L_{\geq 1}$ be a positively graded complete $G$-$sL_\infty$ algebra.
Denote by 
$\Sull(X)$ Sullivan's piece-wise linear de Rham forms on $X$ \cite{Sul77}, 
equipped with the induced $G$-action. 
Let $A$ be a $G$-CDGA model of $X$, that is, a CDGA quasi-isomorphism $A\xrightarrow{\simeq} \Sull(X)$ that is $G$-equivariant. 
Under these assumptions, we prove the following result.
\begin{thm} \label{thm B}\em
There is a natural $G$-equivariant homotopy equivalence of Kan complexes
\begin{equation*}
	\MC\left(A\widehat{\otimes} L\right) \xrightarrow{\simeq} \Map\left(X,\MC\left(L\right)\right).
\end{equation*} 	
In particular, it induces a homotopy equivalence on homotopy fixed points  
$\MC\left(A \widehat \otimes L\right)^{hG} \xrightarrow{\simeq}$

\noindent $ \Map\left(X,\MC(L)\right)^{hG}.$
If furthermore $G$ is finite 
and $A$ is concentrated in degrees strictly smaller than the connectivity of $L$, 
then there is  a natural homotopy equivalence
$$
	\MC\left(\left(A \otimes L\right)^G\right) \xrightarrow{\simeq}\Map\left(X,\MC(L)\right)^{hG}.$$
\end{thm}
The first statement in the result above is Theorem \ref{Model of Map+Conjugation}. 
It  asserts that  $A\widehat{\otimes} L$ is a  $G$-$sL_\infty$-model of the mapping space $\Map\left(X,\MC(L)\right)$ endowed with the conjugation action. 
The second statement, assuming that $G$ is finite, is  Theorem \ref{Models of homotopy fixed points}.
It asserts that  $\left(A\otimes L\right)^G$
is an $sL_\infty$ model for $\Map\left(X,\MC(L)\right)^{hG}$.

\medskip

When $G$ is an infinite group, 
the problems studied in this paper must be approached in a different way.
For such groups, one should not expect a good general relationship between the fixed and homotopy fixed points of $\MC(L)$, see Remark \ref{Rem: Non-finite groups}.

\medskip

It is well-known that the theory of integration or geometric realization of $L_\infty$ algebras can be approached from several perspectives.
Here, we have chosen to work with the Maurer-Cartan simplicial set 
$$\MC : \mathsf{s{\widehat L_\infty}} \to \mathsf{sSet}.$$
There are other alternative constructions: Getzler's $\gamma_\bullet(L)$, Robert-Nicoud--Vallette's  $R(L)$, and Buijs-Félix-Murillo-Tanré's $\langle L \rangle$ (see \cite{Get09,Rob20,Bui20}, respectively).
We explain in Section \ref{Sec: Comparison}
	that our main results also hold for these smaller alternative simplicial sets.

\medskip

\noindent {\bf Conventions:}
In this paper, all the algebraic structures are considered over the field of the rational numbers, but our results hold over any field of characteristic zero.
We use a homological grading for $L_\infty$-algebras.
All topological spaces are assumed to be CW-complexes.
If $V$ is a graded vector space, 
and $G$ is a group that acts linearly on $V$, 
then the graded vector space of fixed points is $$V^G = \left\{v\in V \mid g\cdot v = v \ \forall v\in V\right\},$$
and the graded vector space of coinvariants is $$V_G = V / \textrm{span}\left\{g\cdot v - v \mid g\in G, v\in V\right\}.$$
The linear dual of a graded vector space $V$ is denoted $V^\vee.$

\bigskip

\noindent {\bf Acknowledgments:} The authors are very grateful to Marc Stephan for teaching them about equivariant homotopy theory, and to the Max Planck Institute for Mathematics in Bonn for its hospitality and financial support.
Both authors were partially supported by a Postdoctoral Fellowship of the Max Planck Society while producing this work. The first author has also been partially supported by the MICINN grant PID2020-118753GB-I00 and by the Irish Research Council Postdoctoral Fellowship GOIPD/2019
/823. 
The second author was supported by the Dutch Research Council (NWO) grant number VI.Veni.202.046 and by the Swedish Research Council grant 2019-00536.

\bigskip

\noindent 
\textbf{MSC 2020:} 55P62, 55P91, 55U10.

\noindent 
\textbf{Key words and phrases:} Equivariant rational homotopy theory, L-infinity algebra, Maurer-Cartan simplicial set.

\section{Equivariant rational homotopy theory}
\label{Equivariant Rational}

We recall   the necessary background on $G$-equivariant spaces and simplicial sets in Section \ref{Equivariant}. 
Then, we explain how the Maurer-Cartan simplicial set extends to the $G$-equi\-var\-iant setting in Section \ref{MCSimplicial}.

\medskip

Some classical references on the rational homotopy 
type of $G$-spaces
include \cite{Tri82,Goy89,All93,Scu02}.
The relation between fixed and homotopy fixed points for non-discrete group 
actions on rational 
(not necessarily simply-connected) spaces has been studied in \cite{Bui09,Bui15}.

\subsection{Equivariant homotopy theory background}\label{Equivariant}

We will now recall some of the basics of equivariant homotopy theory, for more details see for example \cite{May96}.
Let $G$ be a finite discrete group (most claims in this section work for topological groups, 
but our interest in this work is on finite groups). 
A $G$-space is a CW-complex $X$ together with a continuous left $G$-action $G\times X \to X$ permuting the cells. 
The \emph{fixed points} for the $G$-action on $X$ is the subspace $$X^G = \left\{ x\in X \mid g\cdot x = x \ \textrm{ for all }g\in G\right\}.$$ 
An equivariant map between $G$-spaces is a continuous map  $f:X\to Y$ which commutes with the $G$-action: 
$f(g\cdot x)=g\cdot f(x)$ for every $x\in X$ and $g\in G$. This defines the category of $G$-spaces. We  also work with the category   $\GsSet$ of $G$-simplicial sets, whose objects are simplicial sets $X_\bullet$  together with left actions $G\times X_n \to X_n$ on the $n$-simplices $X_n$ of $X_\bullet$ compatible with the face and degeneracy maps. A simplicial map is equivariant if it respects the $G$-actions. What we explain in this section for $G$-spaces works as well for $G$-simplicial sets, and vice versa. For brevity, we spell out the details only in one case. If $X$ and $Y$ are $G$-spaces, then the product  $X\times Y$ is a $G$-space with the diagonal action $g \cdot (x,y) = \left(g \cdot x, g \cdot y\right)$.  
The usual mapping space $\operatorname{Map}\left(X,Y\right)$, 
endowed with the compact-open topology, 
becomes a $G$-space with the \emph{conjugation action}  $G \times \Map(X,Y) \to \Map(X,Y)$.
It is given by $(g,f) \mapsto g\cdot f$ as below:
\begin{equation*}
g\cdot f :X\to Y, \quad x \mapsto g \cdot \left(f \left(g^{-1} \cdot x\right) \right) \quad \textrm{ for every } x\in X.
\end{equation*} 
This $G$-action on $\operatorname{Map}\left(X,Y\right)$ is adjoint to the product of $G$-spaces with the diagonal action: 
For any $G$-spaces $X,Y$ and $Z$, there is a natural homeomorphism of $G$-spaces
\begin{equation}\label{Equivariant Tensor-Hom}
\Map\left(X\times Y,Z\right)\cong\Map\left(X,\Map\left(Y,Z\right)\right).
\end{equation}
The set of $G$-equivariant maps  for the corresponding $G$-actions, denoted $\Map_G\left(X,Y\right)$, is defined as the set of fixed points of $\Map\left(X,Y\right)$ with the conjugation action
$$\Map_G\left(X,Y\right) := \Map\left(X,Y\right)^G.$$
This set carries the subspace topology as a subset of the usual mapping space $\operatorname{Map}\left(X,Y\right)$.
The space $X^G$ can be canonically identified with the equivariant mapping space $\Map_G \left(*,X\right)$.
For $G$-simplicial sets, we use the same notation, $\operatorname{Map}_G\left(X_\bullet, Y_\bullet\right)$ and $\operatorname{Map}\left(X_\bullet,Y_\bullet\right)$, for the corresponding simplicial mapping spaces, which we carefully explain next. 
Denote by  $\Delta_\bullet ^n$ the standard $n$-simplex as a simplicial set.
Whenever we consider $\Delta_\bullet^n$ as a $G$-simplicial set for some group $G$, we assume that the $G$-action on $\Delta_\bullet^n$ is trivial.
The mapping simplicial sets have $n$-simplices  
\begin{equation*}
\Map\left(X_\bullet,Y_\bullet\right)_n = \mathsf{sSet} \left(\Delta^n_\bullet \times X_\bullet, Y_\bullet\right), \quad \textrm{ and } \quad \Map_G\left(X_\bullet,Y_\bullet\right)_n = \mathsf{sSet}_G \left(\Delta^n_\bullet \times X_\bullet, Y_\bullet\right).
\end{equation*} The face and degeneracy maps on an $n$-simplex $f$ are respectively given by
\begin{equation*}
d_i\left(f\right) : \Delta_\bullet^{n-1} \times X_\bullet \to Y_\bullet, \quad d_i\left(f\right) = f \circ \left(\delta^i\times \id\right), \quad \textrm{ and } \quad s_j\left(f\right) : \Delta_\bullet^{n+1} \times X_\bullet \to Y_\bullet, \quad s_j\left(f\right) = f \circ \left(\sigma^j\times  \id\right), 
\end{equation*} 
where $\delta^i: \Delta_\bullet^{n-1}\to \Delta_\bullet^{n}$ is the simplicial map induced by the standard $i$-th codegeneracy map of the standard simplex, 
and 
$\sigma^j: \Delta_\bullet^{n+1}\to \Delta_\bullet^{n}$ is the simplicial map induced by the $j$-th coface map of the standard simplex.
In the same way as for $G$-spaces, the $G$-equivariant simplicial maps are the $G$-fixed points for the natural conjugation action on $\Map\left(X_\bullet,Y_\bullet\right)$.
Recall that a \emph{Kan complex} is a simplicial set satisfying the horn-fillers condition, and that if $Y_\bullet$ is a Kan complex, then for any simplicial set $X_\bullet$, the simplicial mapping space $\Map\left(X_\bullet,Y_\bullet\right)$ is also a Kan complex \cite[Proposition 1.17]{Cur71}. Sometimes, we will omit the subscript from the notation of a simplicial set $X_\bullet$, and simply write $X$.\\

 The homotopy theory of $G$-spaces is subtle, and there are several choices for the weak equivalences, see \cite{Bre72,May96}. 
 A (strong) homotopy equivalence of $G$-spaces is an equivariant map $f:X\to Y$ such that for every subgroup $H\leq G$, the induced map between $H$-fixed points $f^H:X^H \to Y^H$ is a weak homotopy equivalence. This condition is usually hard to check, and for us it is enough to use the following weaker notion.
 \begin{definition} 	
 	An \emph{equivariant weak equivalence} is a $G$-equivariant map $f:X\to Y$ that becomes a weak equivalence of spaces when forgetting the group action.
 	That is, $f$ is a $G$-equivariant map that induces an isomorphism on all homotopy groups for every choice of base point, and a bijection on $\pi_0$.
 \end{definition}

\medskip

The fixed points for a given $G$-action on $X$ are not homotopically well behaved. This is solved by introducing the \emph{homotopy fixed points} as the space of equivariant maps $$X^{hG} := \Map_G \left(EG,X\right).$$ 
Here, $EG$ is a contractible space with a free $G$-action. 
The space of orbits of $EG$, denoted $BG$, 
is the \emph{classfying space} of $G$, and there is a well-known principal $G$-bundle  
$$G \hookrightarrow EG \to BG.$$ 
These constructions depend on the choice  of model for $EG$. 
An explicit simplicial model for $EG$ is described in \cite[6.14]{Cur71} (it is called $WG$ there). Since it will be used, we recall it next. The $n$-simplices of $EG$ are
\begin{equation*}
	 EG_n= G^{\times (n+1)}.
\end{equation*} The face and degeneracy maps are given by 
\begin{align*}
	d_i \left(g_{n},...,g_0\right) &= \left\{ \begin{array}{lcl}
	\left(g_{n},...,g_1\right) & \mbox{ if } & i=0 \\
	\left(g_{n},...,g_{i+1},g_{i}\cdot g_{i-1}, g_{i-2},...,g_0\right) & \mbox{ if } & 1\leq i \leq n
	\end{array}
	\right. \\[0.2cm]
	s_j \left(g_{n},...,g_0\right) &= \left(g_{n},..., g_{n-j},e,g_{n-j-1},...,g_0\right)  \quad  \textrm{ for all } \quad  0\leq j \leq n
\end{align*} Here, $e$ denotes the unit element of $G$. The group $G$ acts on the left on $EG$ by $$g \cdot (g_{n},...,g_0)=(g \cdot g_{n},g_{n-1},....,g_0).$$ The $G$-simplicial set described above carries a free $G$-action and can be seen to be contractible, thus providing an explicit model of $EG$. The collapse map $EG \xrightarrow{\simeq} *$ induces an inclusion $X^G \hookrightarrow X^{hG}$ which is usually far from being a homotopy equivalence. For any $G$-space $X$, the homotopy fixed points are the usual fixed points of the conjugation action on the  mapping space: $X^{hG}=\Map\left(EG,X\right)^G$.



\medskip

The following well-known result will be used at several later places.

\begin{proposition}{\emph{\cite[Prop. 1.8]{Car92}}} \label{Prop: w.e. + G map implies w.e. homotopy fixed points}
	Let $f:X\to Y$ be a weak equivalence of simplicial sets which is $G$-equivariant. 
	Then the induced map on homotopy fixed points is a weak equivalence, $f^{hG}:X^{hG}\xrightarrow{\simeq} Y^{hG}$.
\end{proposition}

Let $G$ be a finite group.
Regard $G$ as a simplicial set, denoted by the same letter, by declaring the set of $n$-simplices $G_n$ to be $G$, for all $n$, 
and all face and degeneracy maps are the identity map.
Denote by $\partial \Delta^n$ the boundary of the simplicial set $\Delta^n$, and by $\Lambda^n_k$ its $k$th horn. 
Endow $\Delta^n, \partial \Delta^n$ and $\Lambda_k^n$, for $k=0,1,...,n$, with the trivial $G$-action. 
The following statement, not exactly phrased as we do, can be found at several places. See for instance \cite{May96,Goe99,Ber03}.

\begin{theorem}
	For a finite group $G$, there is a cofibrantly generated model category structure on $\GsSet$ such that:
	\begin{enumerate}
		\item the weak equivalences are the $G$-equivariant simplicial maps $f:X\to Y$ 
		whose underlying map of simplicial sets is a weak equivalence when forgetting the $G$-action,
		\item the fibrations are the $G$-equivariant simplicial maps $f:X\to Y$ 
		whose underlying map of simplicial sets is a Kan fibration when forgetting the $G$-action, and
		\item the cofibrations are determined by the lifting property against acyclic fibrations.
	\end{enumerate}
	The canonical maps $\partial \Delta^n \times G \hookrightarrow \Delta^n \times G$, for $n\geq 0$, form a set of generating cofibrations, 
	and the canonical maps $\Delta_k^n \times G \hookrightarrow \Delta^n \times G$, for $n\geq 0$, form a set of generating acyclic cofibrations. 
\end{theorem}

\begin{proof}
	The details on this result can be found in the mentioned references, but we sketch the main idea.
	The proof is an application of the classical transfer principle for cofibrantly generated model category structures \cite[Thm. 11.3.2]{Hir09} to the pair of adjoint functors $(-)\times G : \sSet \leftrightarrows \GsSet : U$. 
	Here, $(-)\times G$ is the left adjoint, mapping a simplicial set $X$ to the cartesian product $X\times G$, endowed with the diagonal $G$-action considering the trivial $G$-action on $X$, and $U$ is the forgetful functor.
\end{proof}

With the model structure described above, a $G$-simplicial set $Y$ is fibrant if the underlying simplicial set $Y$ is a Kan complex.
From this fact, we immediately conclude the following assertion, 
which will be used later.

\begin{proposition}
	\label{Prop: G Mapping spaces are Kan complexes}
	Let $G$ be a finite group, and $X,Y\in \GsSet$. 
	If $Y$ is fibrant, then the underlying simplicial sets of the mapping spaces $\Map\left(X,Y\right)$ and $\Map_G\left(EG,Y\right)$ are Kan complexes.
\end{proposition}

\subsection{The Maurer-Cartan $G$-simplicial set}\label{MCSimplicial}

 In this section, 
 we extend the Maurer-Cartan simplicial set functor $\MC: \mathsf{s \widehat L_\infty} \to \mathsf{sSet}$ to the $G$-equivariant setting, producing a $G$-simplicial set out of a complete shifted  $L_\infty$-algebra $L$ endowed with a $G$-action, 
 $$\MC: \GL \to \GsSet.$$ 
 We explain the involved categories of shifted $L_\infty$ algebras in this section. The only novelty here is regarding $G$-actions, so our brief explanations are meant mainly to set up the notation. 
 An up-to-date reference for the functor $\MC$ is for example \cite{Rog20}.  
 We chose to work with \emph{shifted} $L_\infty$-algebras for two reasons. First, the sign convention is simpler, and second, several algebraic structures of Lie type appear with a natural shift in degree. Two examples of naturally shifted structures are the Whitehead product on homotopy groups and the Gerstenhaber bracket on Hochschild cohomology. It is well-known that desuspending the underlying chain complex of a shifted $L_\infty$ algebra gives a (traditional) $L_\infty$ algebra.  
 
 \medskip 

A \emph{shifted $L_\infty$ algebra}, or $sL_\infty$ algebra, is an algebra over $\Omega\mathsf{Com}^\vee$,  the cobar construction of the cooperad dual to the commutative operad. Equivalently, it is a graded $\Q$-vector space $L=\left\{L_n\right\}_{n\in \Z}$ together with  degree $-1$ (graded) symmetric linear maps $\ell_i:L^{\otimes i}\to L$, for $i\geq 1$, satisfying the \emph{generalized Jacobi identities} for every $n\geq 1$: 
\begin{align*}
 \sum_{k=1}^{n}\sum_{\sigma\in Sh(k,n-k)} (-1)^{\varepsilon(\sigma,x_1,...,x_n)} \ell_{n-k+1}\left(\ell_k\left(x_{\sigma (1)},...,x_{\sigma (k)}\right),x_{\sigma(k+1)},...,x_{\sigma(n)}\right) = 0.
\end{align*}
Here, $Sh(k,n-k)$ are the $(k,n-k)$ shuffles: 
Those permutations $\sigma$ of $n$ elements such that
\begin{center}
$\sigma(1)<\cdots < \sigma(k)$ \quad and \quad $ \sigma(k+1)<\cdots < \sigma(n),$
\end{center}  and the following sign is taken over all inversions $i<j$ of $S_n:$ $$\varepsilon(\sigma,x_1,...,x_n) = \prod_{i<j} |x_i||x_j|.$$  
Furthermore, an $sL_\infty$ algebra structure on $L$ is also equivalent to a degree $-1$ codifferential $\delta$ on the reduced (i.e. non-counital) cofree conilpotent cocommutative  coalgebra $\bar SL$ generated by $L$. This is done by identifying the homogeneous  component $\delta_n$ of wordlength $n$ of $\delta$ with the coderivation induced by $\ell_n$. From this point of view, the equation $\delta^2 = 0$ corresponds to the generalized Jacobi identities. 
We typically write just $L$ for an $sL_\infty$ algebra, 
omitting the higher structure maps from the notation.

\medskip

All the definitions above give equivalent objects, 
but there
are different notions of morphisms. 
An \emph{$sL_\infty$ morphism} $F:L\to L'$ is a dg coalgebra map between the corresponding dg coalgebras, $$F : \left(\bar SL,\delta\right) \to \left(\bar SL',\delta'\right).$$ 
For $k\geq 1$, denote $S^kL$  the word-length $k$ elements of $\bar SL$, i.e. 
the space of coinvariants $\left(L^{\otimes k}\right)_{S_k}$ under the usual $S_k$-action on $L^{\otimes k}$.  Since a morphism to a cofree coalgebra is determined by its corestriction, 
there is a sequence of symmetric maps arising from $F$,
$$f_k : L^{\otimes k} \to L', \quad k\geq 1,$$
satisfying a family of identities that correspond to the single identity $F\delta=\delta'F.$ 
The $sL_\infty$ morphisms are also called $\infty$-morphisms or shifted $L_\infty$-morphisms.

A different kind of morphisms arises from the fact that   $sL_\infty$ algebras are algebras over an operad. These are called \emph{strict} $sL_\infty$ morphisms.  These morphisms are a particular and much simpler form of $\infty$-morphisms.
They are given by a single degree $0$ linear map $f:L\to L'$ that commutes with all higher brackets, i.e. $$f\ell_n\left(x_1,...,x_n\right) = \ell_n\left(f(x_1),...,f(x_n)\right).$$
These morphisms are precisely the $\infty$-morphisms all of whose components $f_n$ vanish for $n\geq 2$.
The category of  $sL_\infty$ algebras with $\infty$-morphisms is denoted by  $\mathsf{s L_\infty}$. The corresponding category with strict morphisms is denoted $\mathsf{s L_\infty^{St}}$. 
There is an obvious embedding $\mathsf{s L_\infty^{St}} \hookrightarrow \mathsf{s L_\infty}.$

\medskip

To obtain a satisfactory homotopy theory for $sL_\infty$ algebras, 
we must consider complete filtered objects and morphisms.
An  $sL_\infty$ algebra $L$ is \emph{filtered} if it comes equipped with a descending filtration
$$L=F^1 L \supseteq F^2 L \supseteq \cdots $$   such that for all $n\geq 1$, and $i_1,...,i_n$, 
\begin{equation}\label{Filtration}
	\ell_n\left(F^{i_1}L,...,F^{i_n}L\right)\subseteq F^{i_1+\cdots +i_n}L.
\end{equation}
In general, the particular filtration we equip an $L_\infty$ algebra with is not so important; 
the crucial fact is the topology induced by the filtration.
There is a canonical filtration, called the \emph{lower central series}, which is defined for all $sL_\infty$ algebras.
It is given by the intersection of all possible descending filtrations as defined above, 
whose elements in filtration degree $p$ are those elements of $L$ that can be written by bracketing in any form at least $p$ elements of $L$.  
An $sL_\infty$ algebra $L$ is \emph{nilpotent} if its lower central series eventually vanishes, i.e., if there is some $p$ for which $F^pL = 0$; 
and \emph{degree-wise nilpotent} if its lower central series eventually vanishes \emph{degree-wise}, i.e., if for every $k$ there is some $p=p(k)$ such that $\left(F^pL\right)_k = 0$. 
An $sL_\infty$ morphism $f:L\to L'$ is \emph{filtered} if $$f_k\left(F^{i_1}L,...,F^{i_k}L\right) \subseteq 
F^{i_1 + \cdots + i_ k}L'.$$

A filtered $sL_\infty$ algebra $L$ is \emph{complete} if  the canonical map to its completion is an isomorphism, $$L\xrightarrow{\cong} \widehat L = \varprojlim  L/F^nL.$$ 

The category of complete filtered $sL_\infty$ algebras together with filtered $\infty$-morphisms is denoted $\mathsf{s \widehat L_\infty}$. The category of complete filtered  $sL_\infty$ algebras with filtered strict $L_\infty$ morphisms is denoted $\mathsf{s \widehat L_\infty^{St}}$.
There are obvious embeddings 
\begin{center}
	\begin{tikzcd}
	\mathsf{s  L_\infty^{St}} \arrow[r, hook]                         & \mathsf{s  L_\infty}                         \\
	\mathsf{s \widehat L_\infty^{St}} \arrow[r, hook] \arrow[u, hook] & \mathsf{s \widehat L_\infty} \arrow[u, hook]
\end{tikzcd}
\end{center}

Every   $sL_\infty$ algebra 
in this paper will be assumed to be 
filtered and complete, and every $sL_\infty$ morphism will be assumed to be filtered.
That is, we work in the category $\mathsf{s \widehat L_\infty}$.

\medskip

Let $G$ be a group. 
We denote by $G\textrm{-}\mathsf{s\widehat L_\infty}$ the category of complete $sL_\infty$-algebras $L=\left(L,\{ \ell_n\}\right)$ equipped with linear $G$-actions $G\times L_n \to L_n$ compatible with the brackets, 
\begin{equation}\label{ActionOnBrackets}
	g \cdot \ell_n\left(x_1,...,x_n\right) = \ell_n \left(gx_1,...,gx_n\right), 
\end{equation} and with the filtration, $G\cdot F^pL\subseteq F^pL$,  together with filtered $sL_\infty$-morphisms that commute with the $G$-action. That is, equation $(\ref{Filtration})$ holds, and if $f=\left\{f_n:L^{\otimes n}\to L' \right\}$ is an $sL_\infty$-morphism, then for every $n\geq 1$ and $g\in G$,
\begin{equation}\label{ActionOnMorphisms}
	g \cdot f_n\left(x_1,...,x_n\right) = f_n \left(gx_1 ,..., gx_n\right).
\end{equation} Such a $\GL$-algebra is the same thing as a complete  $sL_\infty$-algebra $L$ together with a group homomorphism $G\to \operatorname{Aut}_{\mathsf{s \widehat L_\infty}}\left(L\right)$.
Sometimes, we will refer to these objects by (complete) $G$-$sL_\infty$-algebras. So defined, such $G$-$sL_\infty$-algebras are the inverse limit of a tower of nilpotent $G$-$sL_\infty$-algebras, 
whose fibers at each level are abelian $G$-$sL_\infty$ algebras. 

\medskip

The category $\GCDGA$ has objects the commutative differential graded algebras (CDGA's) 
endowed with a linear left $G$-action compatible with the product and differential.
That is, if $\left(A,d\right) \in \GCDGA$, 
then $G$ acts linearly on $A$ as a graded vector space so that for all $a,b\in A$ and $g\in G$,
\begin{equation*}
	g\cdot (ab) = (g\cdot a)(g\cdot b) \quad \textrm{and} \quad g\cdot (da) = d\left(g\cdot a\right).
\end{equation*}
Here, we are denoting by juxtaposition $ab$ the product of the elements $a,b\in A$.
The morphisms in $\GCDGA$ are the equivariant CDGA maps, i.e, CDGA maps $f:A\to B$ such that $f(g\cdot x)= g\cdot f(x)$ for all $g\in G$ and $x\in A$.

\medskip

If $L\in \GL$ and $A\in \GCDGA$, then the complete $sL_\infty$ algebra  
$$L\widehat \otimes A = \varprojlim (L/F^nL)\otimes A$$ carries the natural $G$-action induced as follows. 
First, we put on each quotient $L/F^nL$
the $G$-action induced from the $G$-action of $L$. 
Then, we put the diagonal action on each tensor  $L/F^nL\otimes A$.
Finally, we consider on $L\widehat \otimes A$
the induced $G$-action on the inverse limit.

Recall that the differential and higher brackets on $L\widehat \otimes A$ are given by
\begin{align*}
	\ell_1\left(x\otimes a\right) &= \ell_1(x)\otimes a + (-1)^{|x|}x\otimes d_Aa \\[0.2cm]
	\ell_k\left( x_1\otimes a_1,...,x_k\otimes a_k \right) &= (-1)^\varepsilon \ell_k\left(x_1,...,x_k\right)\otimes a_1\cdots a_k
\end{align*}
where $\varepsilon=\sum_{i>j} |x_i||a_j|$. 
Denote by  $\Omega_n=A_{PL}(\Delta^n)$ the CDGA of the Sullivan's polynomial de Rham forms on the standard simplices, and by 
$\Omega_\bullet = \left\{\Omega_n\right\}_{n\geq 0}$ the simplicial CDGA obtained by putting all those CDGA's together, taking into account their induced face and degeneracy maps \cite[Chp. 10]{Yve12}.
Endow these CDGA's with the trivial $G$-action.
Then,
each $L\widehat{\otimes }\Omega_n$ 
is a complete $\GL$-algebra, and  $L\widehat{\otimes }\Omega_\bullet$ is a simplicial  $\GL$-algebra. 

\medskip

The \emph{Maurer-Cartan set} of a complete $sL_\infty$-algebra $L$, denoted $MC(L)$, are the degree $0$ elements $z$ satisfying the equation $$\sum_{k\geq 1} \frac{1}{k!} \ell_k\left(z,...,z\right)=0.$$ The \emph{Maurer-Cartan simplicial set} $\MC(L)$ of an $L_\infty$-algebra $L$ has $n$-simplices $$\operatorname{MC}_n\left(L\right) = MC\left(L\widehat{\otimes }\Omega_n \right).$$
The face and degeneracy maps are induced by those of the simplicial CDGA $\Omega_\bullet.$ 
The simplicial set $\MC(L)$ is always a Kan complex \cite{Get09}. 
The functor $$\MC:\GL \to \GsSet $$ is defined as the usual Maurer-Cartan functor on objects, 
and it carries the $G$-action $$g\cdot z = \sum_i (gx_i)\otimes a_i,$$ for an $n$-simplex $z=\sum_i x_i\otimes a_i \in MC\left(L\widehat{\otimes} \Omega_n\right)$ and $g\in G$. Equation (\ref{ActionOnBrackets}) ensures that $g\cdot z\in \MC(L).$ 
With this definition on objects, the $G$-action on $L$ extends to an equivariant version of $\MC$ on morphisms. 
That is, $\MC\left(f\right)$ is given on morphisms as usual, see \cite[Equation (5.42))]{Rog20}.
The equivariance of $\MC(f)$ follows from equation (\ref{ActionOnMorphisms}).

\medskip
In the non-equivariant case, the homotopy groups of $\MC(L)$ were computed in \cite[Theorem 1.1]{Ber15}). There, it is shown that for a complete $sL_\infty$-algebra $L$, the homotopy groups of $\MC(L)$ based at a Maurer-Cartan element $\tau$ are given by 
\[
\pi_{*}\left(\MC(L),\tau\right)\cong H_{*}\left(L^{\tau}\right),
\]
where $L^{\tau}$ denotes the $sL_\infty$-algebra $L$ twisted and truncated  by the Maurer-Cartan element $\tau$. The group structure on $H_1\left(L^{\tau}\right)$ is given by the Baker-Campbell-Hausdorff formula. 
This result will be used later on.
Since we work with $sL_\infty$ algebras concentrated
in positive degrees, (which corresponds to a non-negatively graded traditional $L_\infty$ algebra) there is a canonical Maurer-Cartan element, namely the zero element 
in degree $0$. 
We fix once and for all this base point in all homotopy groups appearing in this paper. 

\medskip

For a given linear $G$-action on a graded vector space $V$, 
the graded subspace of $G$-invariants is denoted by $V^G$. 
If furthermore $V$ is a chain complex, 
then the homology $H_*(V)$ inherits a natural $G$-action by the formula $g\cdot [x] = [g\cdot x]$.
The following well-known result will be used at several places.

\begin{proposition}\label{Proposition1}
	Let $G$ be a finite group and $(V,d)$ a chain complex over $\Q$. 
	Then there is an isomorphism $H_*(V^G)\cong H_*(V)^G.$
\end{proposition}
\begin{proof}
	The inclusion $V^G \hookrightarrow V$ is a chain map that induces a map 
	$i_* : H_*\left(V^G\right)\to H_*(V)$ in homology.
	The image of $i_*$ is $H_*(V)^G$.
	Furthermore, $i_*$ is injective, because if $x\in \Ker(i_*)$ and $a\in V$ is such that $da=x$, then the average $\bar a:= \frac{1}{|G|}\sum_{g\in G} g\cdot a$ is an element of $V^G$ which bounds $x$. 
\end{proof}

Next, we briefly explain the connection of $\MC$ with the $G$-equivariant Sullivan realization functor described in Goyo's thesis \cite{Goy89}.
This connection explains in a precise sense how the results in this paper extend those of Goyo.
In \textit{loc. cit.}, 
the classical adjoint pair between simplicial sets and CDGA's \cite{Bou76} was extended to include a finite group action:
\begin{equation}\label{Ecu:Adjoint Sullivan}
	\Sull \ :\  \GsSet\leftrightarrows \ \GCDGA\ : \langle - \rangle.
\end{equation}
Actually, in \cite{Goy89},
the adjoint pair above is restricted to finite type 1-reduced simplicial sets
and 
finite type simply-connected  CDGA's.
Now, assume $L=L_{\geq 1}$ is a finite type positively graded $sL_\infty$-algebra. Then, via the cochains functor,
	 $L$ uniquely corresponds to a CDGA $\mathcal C^*\left(L\right)=\left(\Lambda V,d\right)$, which is a Sullivan algebra with $V=L^\vee$, and whose differential $d$ has homogeneous component $d_k$ corresponding to the higher bracket $\ell_k$. 
The $G$-action passes through this correspondence via $$\left(g \cdot x\right)^\vee = g^{-1} \cdot \left(x\right)^\vee.$$ Thus, $\left(\Lambda V,d\right)$ is a $\GCDGA.$ 
Combining our discussion so far with \cite[Theorem 2.3]{Ber15},
one can give the following much more precise statement.

\begin{proposition}
	Let $G$ be an arbitrary group.
	Then the Chevalley-Eilenberg cochains functor $\mathcal C^*$ induces an isomorphism of categories 
	\begin{center}
\begin{tikzcd}
	{\left(\left\{\substack{\textrm{finite type $G$-$sL_\infty$ algebras }\\ L=L_{\geq 1}}\right\}, \substack{\infty-\textrm{morphisms }\\ \textrm{compatible with $G$}} \right)} \arrow[rr,"\cong"] &  & {\left(\left\{
		\substack{\textrm{quasi-free finite type $G$-CDGA's}\\ 
		(\Lambda V,d), V=V^{\geq 1}}\right\}, {\tiny{\textrm{$G$-CDGA maps}}}\right).}
\end{tikzcd}
	\end{center} Furthermore, in the correspondence above, $L$ is degree-wise nilpotent if, and only if, $\mathcal C^*(L)$ is a Sullivan algebra.
\end{proposition}

The Maurer-Cartan simplicial set functor $\MC$ is naturally equivalent to Sullivan's realization functor in the $1$-reduced finite type case 
\cite[Proposition 1.1]{Get09}, 
and it can be seen that the $G$-action passes through. Thus, there is a natural isomorphism of simplicial sets which is 
$G$-equivariant: $$\MC\left(L\right)\cong \left\langle \mathcal C^*\left(L\right)\right\rangle.$$ 
Therefore, 
the results in this paper extend those of Goyo, 
since we work with a broader class of spaces: 
connected, not necessarily of finite rational type,
and whose fundamental group is not necessarily trivial - it must be of the form $H_1\left(L\right)$ for $L$ some complete shifted $L_\infty$ algebra endowed with the Baker-Campbell-Hausdorff group structure.

\section{The Sullivan conjecture for the Maurer-Cartan $G$-simplicial set}\label{SullivanConjecture}

For a $G$-space or $G$-simplicial set $X$, determining whether the natural  inclusion $X^G\hookrightarrow X^{hG}$ is a homotopy equivalence after completing at a given prime $p$  is a difficult problem. Its answer in the affirmative goes under the name of the \emph{generalized Sullivan conjecture}, see \cite{Mil84}. In this section, we prove that whenever $G$ is a finite group, $p=0$ and $X$ is of the form $\MC(L)$ for a complete $G$-$sL_\infty$-algebra concentrated in positive degrees, the generalized Sullivan conjecture holds.
The main result is Theorem \ref{HomEq}, corresponding to Theorem \ref{Teo:Main} in the introduction. 

\bigskip

Let $L\in \GL$. The \emph{fixed points of $L$} are 
\begin{equation*}
L^G = \left\{ x \in L \mid gx = x  \ \textrm{ for all } g \in G \ \right\} \subseteq L.
\end{equation*} 
It follows from Equation (\ref{ActionOnBrackets})  that $L^G$ is an $sL_\infty$-subalgebra of $L$. 
Furthermore, $L^G$ inherits the natural filtration $$F^p\left(L^G\right) := F^pL \cap L^G = \left(F^pL\right)^G.$$
With respect to this filtration, $L^G$ is complete, and the inclusion $L^G \hookrightarrow L$ is a filtered morphism.
Recall from Proposition \ref{Prop: G Mapping spaces are Kan complexes}
that if $X,Y\in \GsSet$, 
and $Y$ is a Kan complex, then the mapping spaces $\Map\left(X,Y\right)$ and $\Map_G\left(X,Y\right)$ are Kan complexes. 
Thus, for any complete $G$-$sL_\infty$-algebra $L$, the simplicial sets $\MC(L)^G$ and $\MC(L)^{hG}$ are Kan complexes.

\begin{remark}\label{Remark1}
	The Maurer-Cartan $G$-simplicial set functor commutes with $G$-fixed points.
	That is, there is a commutative diagram of functors 
	\begin{center}
		\begin{tikzcd}
			\GL \arrow[r, "(-)^{G}"] \arrow[d, "\MC"'] & \mathsf{s \widehat L_\infty} \arrow[d, "\MC"] \\
			\GsSet \arrow[r, "(-)^{G}"]                 & \mathsf{sSet}                                
		\end{tikzcd}
	\end{center} 
Thus,
$\MC(L^G)= \MC(L)^{G}.$
\end{remark}

The main result is the following. 

\begin{theorem}\label{HomEq}
		Let $G$ be a finite group and  $L=L_{\geq 1}$  a positively graded complete $G$-$sL_\infty$-algebra. Then, the natural inclusion 
\begin{equation*}
	\MC \left(L\right)^G \hookrightarrow \MC\left(L\right)^{hG}
\end{equation*} is a homotopy equivalence of Kan complexes. 
\end{theorem}

\begin{proof}
	We use a standard strategy for complete $sL_\infty$-algebras (see for example \cite{Get09,Dol15,Ber15}). This consists in first proving the statement for abelian $G$-$sL_\infty$-algebras, and then  using induction over  the tower of fibrations of nilpotent $G$-$sL_\infty$-algebras 
	\begin{center}
			\begin{tikzcd}
				\cdots \arrow[r, two heads] & L/F^nL \arrow[r, two heads] & \cdots \arrow[r, two heads] & L/F^3L \arrow[r, two heads] & L/F^2L
			\end{tikzcd}
	\end{center} of which $L$ is the inverse limit.
The induction argument can be started because $L/F^2L$ is always abelian, 
and iterated because each fiber $F^nL/F^{n+1}L$ is abelian too.

\smallskip

{\noindent \em Step 1: The abelian case.} Let $L$ be abelian.
	We give an explicit simplicial  retraction
	\begin{center}
		\begin{tikzcd}
			\MC\left(L\right)^G \arrow[rr, "i"', shift right] &  & \MC\left(L\right)^{hG} \arrow[ll, "p"', shift right] \arrow["K"', loop, distance=2em, in=125, out=55]
		\end{tikzcd}
	\end{center}
where $pi=\id$ and $K : \MC\left(L\right)^{hG}\times \Delta^1 \to \MC\left(L\right)^{hG}$ is a simplicial homotopy between $ip \simeq \id$. 
That is, $K$ is a simplicial map satisfying $$K\left(-,0\right)=ip \quad \textrm{and} \quad K\left(-,1\right)=\id.$$
Define $p:\MC\left(L\right)^{hG} \to \MC \left(L\right)^G$ on  simplices $f\in \mc_n\left(L\right)^{hG}=\sSet_G\left(\Delta^n\times EG, \MC\left(L\right)\right)$ by
	\begin{equation*}
		\begin{tikzcd}[row sep=tiny]
		p\left(f\right)=f^\Sigma: \Delta^n\times * \ar{r}& \MC\left(L\right)\\
		(\tau, x) \ar[mapsto]{r} & \frac{1}{|G|} \sum_{g\in G}f\left(\tau,\left(g,e,...,e\right)\right).
		\end{tikzcd}
	\end{equation*} Since $L$ is abelian and $G$ is finite, the averaged sum above is indeed a Maurer-Cartan element. Thus, $p$ is well-defined. It is straightforward to check that $p$ is simplicial. Denote by $i$ the natural inclusion $\MC \left(L\right)^G \hookrightarrow \MC\left(L\right)^{hG}$ induced by the collapse map $\pi:EG\xrightarrow{\simeq} *$. A straightforward check gives that $pi$ is the identity on $\MC \left(L\right)^G$ and that that $ip$ is the averaged symmetrization on the first variable. 
That is, for any $f\in \mc_n\left(L\right)^{hG}$, $(ip)(f):\Delta^n\times EG \to \MC\left(L\right)$ is given by $$\left(\tau, g_m,...,g_0\right) \mapsto \frac{1}{|G|}\sum_{g\in G} f(\tau, (g,e,...,e)).$$
Next, we define the homotopy $K$.	
To do so, we first show that given any $G$-simplicial map $f:EG \to \MC(L)$, it is $G$-equivariantly homotopic to its averaged symmetrization on the first variable,
	$$f^\Sigma \equiv f^\Sigma \circ \pi: EG \xrightarrow{\simeq} * \xrightarrow{f^\Sigma} \MC(L).$$ 
	Recall that $f^\Sigma:*\to \MC(L)$ is given by mapping any $k$-simplex as follows: $$x \in * \mapsto \frac{1}{|G|}\sum_{\sigma\in G} f\left(\sigma,e,...,e\right).$$ We therefore give a $G$-simplicial map $H=H(f):EG \times \Delta^1  \to \MC(L)$ with $$H\left(-,0\right)=f^\Sigma \quad \textrm{and} \quad H\left(-,1\right)=f.$$ 
Explicitly, given  $\left(\left(g_m,...,g_0\right),[0,...,0,1,...,1]\right)$ an $m$-simplex  of $EG\times \Delta^1$, define:
	\begin{equation*}
		H\left((g_m,...,g_0), [\underbrace{0,{ \ ...\  },0}_k,\underbrace{1,{ \ ...\  },1}_{m+1-k}]\right) = \frac{1}{|G|} \sum_{\sigma \in G} f \left(g_m,...,g_{k},\sigma,e,...,e\right).
	\end{equation*} 
In the formula above, as comes forced by the top and bottom of the cylinder conditions, we understand that for $k=m$, $$ H\left((g_m,...,g_0), [0,...,0]\right) = \frac{1}{|G|} \sum_{\sigma \in G} f \left(\sigma,e,...,e\right), $$ and that for $k=0$, $$ H\left((g_m,...,g_0), [1,...,1]\right) = f \left(g_m,...,g_0\right).$$ 
This is a $G$-simplicial map. 
Indeed, for $d_0$ we have: 
	\begin{align*}
		d_0H((g_m,...,g_0), &[\underbrace{0,{ \ ...\  },0}_k,\underbrace{1,{ \ ...\  },1}_{m+1-k}])=d_0 \left(\frac{1}{|G|} \sum_{\sigma \in G} f \left(g_m,...,g_{k},\sigma,e,...,e\right)\right) \\ &= \frac{1}{|G|}\sum_{\sigma \in G}  d_0f \left(g_m,...,g_{k},\sigma,e,...,e\right)= \frac{1}{|G|}\sum_{\sigma \in G}  fd_0 \left(g_m,...,g_{k},\sigma,e,...,e\right)\\
		& = \frac{1}{|G|}\sum_{\sigma \in G}  f \left(g_m,...,g_{k-1},\sigma,e,...,e\right) = Hd_0\left(\left(g_m,...,g_1\right), [\underbrace{0,{ \ ...\  },0}_k,\underbrace{1,{ \ ...\  },1}_{m+1-k}]\right).
	\end{align*} For $d_i$, with $0<i\leq k$, we similarly have:
	\begin{center}
		\begin{tikzcd}
			{(g_m,...,g_0), [\overbrace{0,{ \ ...\  },0}^k,\overbrace{1,{ \ ...\  },1}^{m+1-k}]} \arrow[rr, "H"] \arrow[d, "d_i"']             &  & {\frac{1}{|G|} \sum_{\sigma \in G} f \left(g_m,...,g_{k},\sigma,e,...,e\right)} \arrow[d, "d_i"] \\
			{(g_m,...,g_{i+1}\cdot g_{i},...,g_0), [\underbrace{0,{ \ ...\  },0}_{k-1},\underbrace{1,{ \ ...\  },1}_{m+1-k}]} \arrow[rr, "H"] &  & {\frac{1}{|G|} \sum_{\sigma \in G} f \left(g_m,...,g_{i+1}\cdot g_{i},...,g_k,\sigma,e,...,e\right)}                 
		\end{tikzcd}
	\end{center} For $d_i$  with $i> k$, we have a diagram like the one above where the only difference is that the down-left corner has $k$ zeroes and $(m-k)$ ones. Similarly, one checks degeneracies. It is immediate that the $G$-action goes through, since $G$ acts on the left.
	Summarizing, we are giving a map
	\begin{align*}
		H:\mathsf{sSet}_G(EG, \MC(L&)) \longrightarrow \mathsf{sSet}_G\left(EG\times \Delta^1, \MC(L)\right)\\
		&\varphi \longmapsto H_\varphi\ :\ \varphi^\Sigma \simeq \varphi.
	\end{align*} This is an auxiliary map for defining the homotopy $K$ between $ip$ and $\id$ in $\MC(L)^{hG}$; 
it is explicitly given on $n$-simplices by:
	
	\begin{equation*}
\begin{tikzcd}
{K_n:\mathsf{sSet}_G\left(\Delta^n \times EG, \MC(L)\right) \times \Delta^1_n} \arrow[r] & {\mathsf{sSet}_G\left(\Delta^n \times EG, \MC(L)\right)} &                                                      \\
{(f,\tau)} \arrow[r, maps to]                                                                     & K_n(f,\tau):\Delta^n \times EG \arrow[r]                                 & \MC(L)                                                \\
& {\left(\sigma,g_m,...,g_0\right)} \arrow[r, maps to]                  & {H_{f(\sigma,-)} \left((g_m,...,g_0),\tau\right)}
\end{tikzcd}
	\end{equation*}
	
	The map $K$ is $G$-simplicial by construction.\\
	
	{\noindent \em Step 2: Induction.} Because of Step 1, the inclusion $\MC\left(L\right)^G \hookrightarrow \MC\left(L\right)^{hG}$ is a homotopy equivalence if $L$ is abelian. Let $L$ be a complete $G$-$sL_\infty$-algebra. Since the $G$-action preserves the filtration, such an $sL_\infty$-algebra $L$ is the inverse limit of a tower of central extensions of nilpotent $G$-$sL_\infty$-algebras, whose fibers are abelian $G$-$sL_\infty$-algebras:
	\begin{center}
\begin{tikzcd}
	\cdots \arrow[r, two heads] & L/F^{n+1}L \arrow[r, two heads]   & \cdots \arrow[r, two heads] & L/F^4L \arrow[r, two heads] & L/F^3L \arrow[r, two heads] & L/F^2L \\
	& F^nL/F^{n+1}L \arrow[u, hook] &                             & F^3L/F^4L \arrow[u, hook]   & F^2L/F^3L \arrow[u, hook]   &       
\end{tikzcd}
	\end{center} Note that $L/F^2L$ is abelian as well. The induction hypotheses is that the inclusion $\MC\left(L/F^nL\right)^G \hookrightarrow \MC\left(L/F^nL\right)^{hG}$
is a homotopy equivalence. Disregarding the $G$-action, each such a central extension
	\begin{center}
\begin{tikzcd}
	F^nL/F^{n+1}L \arrow[r, hook] & L/F^{n+1}L \arrow[r, two heads] & L/F^{n}
\end{tikzcd}
	\end{center}
	 maps to a fibration of simplicial sets (see for instance \cite[Prop. 4.7]{Get09} or \cite[Thm. 3]{Rog20})
	\begin{center}
\begin{tikzcd}
	\MC \left(F^nL/F^{n+1}L\right) \arrow[r, hook] & \MC \left(L/F^{n+1}L\right) \arrow[r, two heads] & \MC \left(L/F^{n}L\right).
\end{tikzcd}
	\end{center} The fibration above is of $G$-simplicial sets, and taking fixed or homotopy fixed points in it yields again a fiber sequence. Therefore, we can compare base and fiber between fixed and homotopy fixed points. That is, we consider the diagram
	\begin{center}
\begin{tikzcd}
	\MC\left(F^nL/F^{n+1}L\right)^G \arrow[r, hook] \arrow[d, "\simeq"', hook] & \MC\left(L/F^{n+1}L\right)^G \arrow[r, two heads] \arrow[d, , hook] & \MC\left(L/F^{n}L\right)^G \arrow[d, "\simeq", hook] \\
	\MC\left(F^nL/F^{n+1}L\right)^{hG} \arrow[r, hook]                         & \MC\left(L/F^{n+1}L\right)^{hG} \arrow[r, two heads]                         & \MC\left(L/F^{n}L\right)^{hG}                       
\end{tikzcd}
	\end{center} 
There are homotopy equivalences on the fibers (because $F^nL/F^{n+1}L$ is abelian) and on the base spaces (by the induction hypothesis). It remains to show that the middle arrow is a weak equivalence. 
	Since the above is a commutative diagram of fibrations, there is a commutative diagram between the induced long exact sequences of homotopy groups at the canonical base point
\begin{center}
\begin{tikzcd}
 \vdots \arrow[d] & \vdots \arrow[d] \\
 \pi_k\left(\MC\left(F^nL/F^{n+1}L\right)^G\right) \arrow[r] \arrow[d]&\pi_k\left(\MC\left(F^nL/F^{n+1}L\right)^{hG}\right) \arrow[d] \\
\pi_k\left(\MC \left(L/F^{n+1}L\right)^G \right) \arrow[d]\arrow[r] &  \pi_k\left(\MC \left(L/F^{n+1}L\right)^{hG} \right) \arrow[d] \\
 \pi_k\left( \MC \left(L/F^{n}L\right)^G \right) \arrow[d]\arrow[r] & \pi_k\left( \MC \left(L/F^{n}L\right)^{hG} \right) \arrow[d] \\
 \vdots & \vdots
 \end{tikzcd}
\end{center}
An application of the 5-Lemma (see for example \cite[Lemma 3.1]{Yve12}) gives that 
	\[
	 \pi_k\left( \MC \left(L/F^{n+1}L\right)^{G} \right)\rightarrow \pi_k\left( \MC \left(L/F^{n+1}L\right)^{hG} \right)
	\]
is also an isomorphism for all $k \geq 0$. To finish, recall that 
if a map between
towers of fibrations is a levelwise fibration or a weak equivalence, then so is the
induced map on inverse limits \cite[Chapter IV]{Goe99}, which completes the proof.
\end{proof}

The following consequence extends another of the main results of \cite{Goy89}.

\begin{corollary}\label{Homotopy Groups}
	Let $G$ be a finite group and $L=L_{\geq 1}$ be a positively graded complete $G$-$sL_\infty$-algebra. Then, $$\pi_*\left( \MC(L)^{hG}\right) \cong \big(\pi_*\MC(L)\big)^G.$$
\end{corollary}

\begin{proof} Apply, in the order given, the isomorphisms of Theorem \ref{HomEq}, Remark \ref{Remark1},  \cite[Theorem 1.1]{Ber15}, Proposition \ref{Proposition1} 2., and again \cite[Theorem 1.1]{Ber15} to get the following chain of isomorphisms:
	\begin{equation*}
		\pi_*\left(\MC(L)^{hG}\right) \ \cong \ \pi_* \left(\MC\left(L\right)^{G}\right) \ \cong \ \pi_* \left(\MC\left(L^{G}\right)\right)\ \cong \ H_{*} \left(L^G\right) \ \cong \ H_{*}\left(L\right)^G \ \cong \ \big(\pi_*\MC(L)\big)^G.
			\end{equation*}
Here, 
if $*=1$,
the group structure on $H_1(L)$ is given by the Baker-Campbell-Hausdorff formula.
\end{proof}

\begin{remark} 
 Neither Theorem \ref{HomEq}, nor the main result of \cite{Goy89} assuming  $X$ simply-connected of finite rational type, prove that for $G$ finite and $X$ a connected rational space, the inclusion $X^G \hookrightarrow X^{hG}$ is a homotopy equivalence. 
 Indeed, even though a nice enough connected rational space $X$ is of the weak homotopy type of the spatial realization of $\MC(L)$ for some complete $sL_\infty$-algebra $L$, this identification typically does not respect the $G$-action. 
 For instance, we can take $X$ to be a space endowed with a free $G$-action, and any weakly equivalent simplicial set of the form $\MC(L)$ will always have the zero element as a fixed point. 
 One should be careful in choosing the correct notion of the rationalization of a $G$-space, which is outside the scope of this work.
\end{remark}

\begin{remark} 
	\label{Rem: Non-finite groups}
	As mentioned in the introduction, 
	one should not expect a good general relationship between fixed and homotopy fixed points of $\MC(L)$ for infinite groups.
	For instance, let $G=\mathbb{Z}$ be the discrete group of the integers and $L$ any (non-trivial) complete $sL_\infty$ algebra.
	Recall that the classifying space $B\Z$ of the discrete group of the integers is the circle $S^1$.
	Endow $L$ with the trivial $\Z$-action.
	Then, on the one hand, $\MC(L)^{\Z} = \MC(L)$ because the action is trivial. 
	On the other hand, 
	\begin{align*}
		\MC(L)^{h\Z} &= \Map_\Z\left(E\Z,\MC(L)\right) = \Map\left(B\Z,\MC(L)\right) = \Map\left(S^1, \MC(L)\right) = \mathcal L \MC(L)
	\end{align*}
	is the free loop space on $\MC(L)$.
	For other more involved infinite groups, like the circle $S^1$,
	similar manipulations produce straightforward  examples of fixed points and homotopy fixed points of Maurer-Cartan spaces having radically different homotopy type.
\end{remark}

\section{Rational models for fixed and homotopy fixed points}\label{LieModels}

Let $X$ be a $G$-simplicial set and $L$ a complete $G$-$sL_\infty$algebra. 
Under certain connectivity assumptions on $X$ and $L$,
we give in this section:
\begin{itemize}
	\item  a $G$-$sL_\infty$-model for the mapping space $\Map\left(X,\MC\left(L\right)\right)$ endowed with the conjugation action, and
	\item an $sL_\infty$-model for $\Map\left(X,\MC\left(L\right)\right)^{hG}$, the homotopy fixed points of the mapping space above. 
\end{itemize}  
These results are theorems \ref{Model of Map+Conjugation} and  \ref{Models of homotopy fixed points}, respectively, and together they correspond to Theorem \ref{thm B} in the introduction.

\medskip

There are at least two possible approaches for studying $G$-equivariant rational homotopy theory. 
The most complete approach might consist in declaring a rational model of the $G$-homotopy type of $X$ as a functor 
from the homotopy orbit category $\mathcal O_G$ of $G$ to  the category $\GL$. 
This is outside of the scope of this article.
This approach is pursued in \cite{Tri82,Gol98,Scu08},
using Sullivan's perspective of rational homotopy theory and CDGA's. 
An alternative approach, 
not so exhaustive, 
but arguably much more prone to computations, 
consists in producing weaker algebraic models that still capture a part of the $G$-structure.
We stick to this second approach in this work, and will explain the details in the next paragraphs.

We will focus on $G$-simplicial sets (everything works as well for $G$-spaces).
For us, the weak equivalences are the $G$-equivariant simplicial maps inducing an isomorphism in all homotopy groups, 
and a bijection on $\pi_0$. 
To give our algebraic models, we will nonetheless focus on connected simplicial sets.
As is widely known, if the fundamental group of a pointed simplicial set $X$ is not nilpotent, 
then there are different  non-equivalent choices for the rationalization of $X$.
Bearing this in mind,
we stick to the following convention in this section: 
 simplicial sets will be reduced 
(i.e., possessing a single $0$-simplex) and $\Q$-good in the sense of Bousfield-Kan \cite{Bou72}.
With this convention, such a simplicial set $X$ is called \emph{rational} if the canonical map $X\to \Q_\infty(X)$ from $X$ onto its Bousfield-Kan $\Q$-completion is a weak equivalence,
and a map $f:X\to Y$ is defined to be a \emph{rational equivalence}
if the induced map 
$\Q_\infty\left(f\right) : \Q_\infty(X)\to \Q_\infty(Y)$ on the Bousfield-Kan $\Q$-completions is a weak homotopy equivalence \cite{Bou72}. 
A \emph{$G$-$sL_\infty$-model} of such a $G$-simplicial set $X$ is a complete $G$-$sL_\infty$-algebra $L$ for which there is a zig-zag of rational equivalences connecting $\MC\left(L\right)$ and $X$ (equivalently, $\Q_\infty(X)$) formed by equivariant maps, 
\begin{equation}\label{ecu:Zig Zag}
\begin{tikzcd}
X & \cdots \arrow[l, "\simeq"'] \arrow[r, "\simeq "] & \MC\left(L\right).
\end{tikzcd}
\end{equation}
Recall from Equation (\ref{Ecu:Adjoint Sullivan})
the contravariant adjunction between simplicial sets and rational CDGA's given by Sullivan's piece-wise polynomial de Rham forms. As mentioned in Section \ref{MCSimplicial}, this was extended in \cite{Goy89} by taking into account the action of a finite group $G$,
$$\Sull : \GsSet \leftrightarrows \GCDGA : \left\langle-\right\rangle.$$
This adjuntion exists for any group, not necessarily finite.
Recall that for $G$-simplicial sets $X$ and $Y$, $\Map\left(X,Y\right)$ is a $G$-simplicial set with the conjugation action. The following result extends \cite[Theorem 6.6]{Ber15} to the $G$-equivariant setting. 

\begin{theorem}\label{Model of Map+Conjugation}  Let $G$ be an arbitrary group. 
If $X$ is a $G$-simplicial set and $L=L_{\geq1}$ is a positively graded complete $G$-$sL_\infty$-algebra, then there is a natural homotopy equivalence of Kan complexes which is $G$-equivariant
	\begin{equation}\label{mapa}
	\varphi : \MC\left(\Sull\left(X\right)\widehat{\otimes} L\right) \xrightarrow{\simeq} \Map\left(X,\MC\left(L\right)\right).
	\end{equation} 	
In particular, $\varphi$ induces a homotopy equivalence $$\MC\left(\Sull\left(X\right) \widehat \otimes L\right)^{hG} \xrightarrow{\simeq} \Map\left(X,\MC(L)\right)^{hG}.$$
\end{theorem}

\begin{proof}
The map $\varphi$ is analogous to the one defined in \cite[Theorem 6.6]{Ber15}.
Following the proof in \emph{loc. cit.}, 
we see that $\varphi$ is equivariant when considering the functor $\MC$  and endow the mapping space with the conjugation action.
Indeed, $\varphi$ depends only  on the natural quasi-isomorphisms $\Sull(X)\otimes \Sull(\Delta^n) \xrightarrow{\simeq} \Sull\left(X\times \Delta^n\right)$, which are $G$-equivariant, 
and the natural set-theoretic map (in the particular case that $Z=X\times \Delta^n$)
$$\mu:MC\left(\Sull\left(Z\widehat \otimes L\right)\right) \to \mathsf{sSet}\left(Z,\MC(L)\right)$$
sending a Maurer-Cartan element $\tau \in MC\left(\Sull\left(Z\widehat \otimes L\right)\right)$ to the simplicial map $$\mu(\tau) :Z\to \MC(L), \quad z \mapsto z^*(\tau),$$
where we see $z\in Z_n$ as a simplicial map $\Delta^n \to Z$, and then $z^* = \Sull\left(z:\Delta^n \to Z\right)\widehat \otimes \id_L.$
By its definition, 
the map $\mu$ is also compatible with the $G$-action.
The fact that the induced map $\varphi^{hG}$ is a weak equivalence follows from Proposition \ref{Prop: w.e. + G map implies w.e. homotopy fixed points}.
\end{proof}

Recall our definition of $G$-$sL_\infty$-model explained at the beginning of the section.
In this sense, the statement of Theorem \ref{Model of Map+Conjugation}
is telling us that the complete $G$-$sL_\infty$ algebra   $\Sull(X)\widehat{\otimes}L$ is a $G$-$sL_\infty$-model  of $\Map\left(X,\MC(L)\right)$.

\begin{remark}\label{remark: Any G-CDGA model is ok}
In Theorem \ref{Model of Map+Conjugation} above, we may substitute $\Sull(X)$ by any $G$-CDGA model of $X$, and the result still holds. 
Recall that $A$ is a $G$-CDGA model of $X$ if there is a CDGA quasi-isomorphism $A \xrightarrow{\simeq} \Sull(X)$ which is a $G$-map.
	In this case, the induced natural weak equivalence
	$$\MC\left(A\widehat \otimes L\right) \xrightarrow{\simeq } \MC\left(\Sull(X)\widehat \otimes L\right)$$ is $G$-equivariant. Therefore, if $A$ is any $G$-CDGA model of $X$, then $A\widehat \otimes L$ is a $G$-$sL_\infty$-model of the mapping space $\Map\left(X,\MC(L)\right).$ 
\end{remark}

Combining the result above with Theorem \ref{HomEq}, we obtain following result.

\begin{theorem}\label{Models of homotopy fixed points}
Let $G$ be a finite group, 
$A$ a $G$-CDGA model of a $G$-simplicial set $X$,
and $L=L_{\geq 1}$  a positively graded complete $G$-$sL_\infty$ algebra.
If $A$ is concentrated in degrees strictly smaller than the connectivity of $L$, 
then 
then there is a natural homotopy equivalence of Kan complexes 
	\begin{equation*}
	\MC\left(\left(A \otimes L\right)^G\right) \xrightarrow{\simeq}\Map\left(X,\MC(L)\right)^{hG},
	\end{equation*} 
	where the source is the Maurer-Cartan simplicial set of the complete $sL_\infty$-algebra of $G$-invariant elements of $A\otimes L$, and the target is the space of homotopy fixed points of the simplicial mapping space $\Map\left(X,\MC(L)\right)$.
\end{theorem} 

\begin{proof} 
Since $A$ is a $G$-CDGA model of $X$, there is an induced natural weak equivalence
 $$\MC\left(A\widehat \otimes L\right) \xrightarrow{\simeq } \MC\left(\Sull(X)\widehat \otimes L\right)$$ that is $G$-equivariant. 
Under the hypotheses of the statement, the map $\varphi$ in (\ref{mapa}) is also a weak equivalence and a $G$-map.
Therefore, the composition of the two maps yields a  $G$-equivariant weak equivalence
$$\Psi:\MC\left(A\widehat \otimes L\right)
 \xrightarrow{\simeq } \Map\left(X,\MC\left(L\right)\right).$$
The map induced by $\Psi$ on homotopy fixed points is then a weak equivalence (Proposition \ref{Prop: w.e. + G map implies w.e. homotopy fixed points}). 
This, together with Remark \ref{Remark1} and Theorem \ref{HomEq} gives the following sequence of weak equivalences, finishing the proof:
	\begin{equation*}
	 \MC\left(\left(A{\widehat \otimes} L\right)^{G}\right) =\MC\left(A{\widehat \otimes} L\right)^{G} \xhookrightarrow{\simeq} \MC\left(A{\widehat\otimes} L\right)^{hG} \xrightarrow{\Psi^{hG}} \Map\left(X,\MC\left(L\right)\right)^{hG}.
	\end{equation*}
\end{proof}
An $sL_\infty$ model of a reduced $\Q $-good simplicial set $X$ 
(with no $G$-action) is a complete $sL_\infty$ algebra $L$ for which there is a zig-zag as in (\ref{ecu:Zig Zag}), but disregarding $G$-actions.
In this sense, 
Theorem \ref{Models of homotopy fixed points} above is telling us that the complete $sL_\infty$ algebra of $G$-invariants $\left(A{\otimes}L\right)^G$ is an $sL_\infty$-model for $\Map\left(X,\MC(L)\right)^{hG}.$ 

\medskip

We finish the section with some examples.

\begin{example} {\bf{(Spheres with antipodal action)}}
	Endow the connected sphere $S^n$ with the  antipodal action by $\Z_2$. 
	Denote by $\sigma$ the non-trivial transposition of $\Z_2$.
	A $G$-$sL_\infty$ model of this $G$-space is the following.
	\begin{itemize}
		\item For odd $n$,
		$$L = \Q  x, \quad  \quad |x|=n, \qquad  \sigma:x\mapsto x.$$
		In this case, $L^G = L$ is the trivial $sL_\infty$ algebra. Therefore, $\MC(L)^{hG}\simeq K\left(\Q,n\right).$	
		\item For even $n$,
		\begin{equation*}
			L =  \Q  x \oplus \Q  [x,x]. 
		\end{equation*} 
Here, we also have $|x|=n$ and $\sigma : x\mapsto -x.$
The action of $\sigma$ extends trivially to the bracket $[x,x]$.
Therefore, $L^G = \Q  [x,x]$ is abelian $1$-dimensional. Thus, $\MC(L)^{hG} = K\left(\Q,2n-1\right).$
\end{itemize}
\end{example}

\begin{example} 
	{\bf{($\Z_2$ actions on the complex projective spaces)}}
	Let $\C P^n$ be the complex projective $n$-dimensional space.
	It has an $sL_\infty$ model 
	$$L = \Q   x\oplus \Q  y, \qquad |x| = 2, \quad  |y| = 2n+1,$$ 
	with a single non-vanishing higher bracket $$\ell_{n+1}(x,...,x) = \frac{1}{(n+1)!} y.$$
	Any $\mathbb{Z}_2$-action on this complete $sL_\infty$-algebra acts by scalar multiple on the linear generators $x$ and $y$.
	Thus, it is of the form 
	 $x\mapsto ax$ and $y\mapsto by$ for some $a,b\in  \mathbb{Q}$.
	Imposing compatibility of this linear action with the only non-trivial bracket $\ell_{n+1}$, 
	we deduce that $b=a^{n+1}$.
	Therefore, each non-zero rational  $a\in \mathbb{Q}^*=\mathbb{Q}-\{0\}$ determines a distinct $\mathbb{Z}_2$-action on the complete $sL_\infty$ model of $\mathbb{C}P^n$, and these are all the possible $\mathbb{Z}_2$-actions.	
Let us compute the homotopy type of the corresponding homotopy fixed points spaces after applying the functor $\MC$ to all these possibilities. 	
Fixed $a\in \mathbb{Q}^*$,
	 we find that:
	\begin{itemize}
		\item If $a=1$, then the action is trivial.
		
		\item If $a=-1$, then $x\mapsto -x$ and $y\mapsto (-1)^{n+1}y.$ Therefore:
		\begin{itemize}
			\item If $n$ is even,
			 then $x\mapsto -x$ and $y \mapsto -y.$ 
			In this case, we find that $$\MC(L)^{hG} \simeq \MC(L^G) = \MC(0) = *.$$
			
			\item If $n$ is odd,  then $x \mapsto -x$ and 
			$y \mapsto y.$
			Therefore, $L^G = \langle y \rangle$ and we find that $$\MC(L)^{hG} \simeq \MC(L^G) = \MC(\langle y \rangle) = K\left(\Q,2n+1\right).$$
		\end{itemize}
	\item If $a\neq \pm1$, then $x\mapsto ax$ and $y\mapsto a^{n+1}y$. In this case, $L^G=0$ and therefore, $\MC(L)^{hG} \simeq *.$
	\end{itemize}
Geometrically, there is a sequence of inclusions of CW-complexes 
$$
* \subseteq \mathbb{C}P^1 \subseteq \mathbb{C}P^2 \subseteq \cdots \subseteq  \mathbb{C}P^n \subseteq \cdots \subseteq \mathbb{C}P^\infty=K(\mathbb{Z},2).
$$
There is a $\mathbb{Z}_2$-action on $K(\mathbb{Z},2)$ sending the generating 2-cell $e$  to $\sgn(\sigma)e$,
and this action restricts to each skeleton, producing a sequence of $\mathbb{Z}_2$-spaces.
The rationalization of the sequence above is modeled by a sequence of complete $sL_\infty$ algebras,
each of whom models the corresponding complex projective space and is equipped with a $\mathbb{Z}_2$-action which is precisely the case $a=-1$ above.			
\end{example}	

\begin{example}{\bf{(Products of Eilenberg-Mac Lane spaces with symmetric group action)}}
	Let $L = \Q u_1 \oplus \cdots \oplus \Q  u_m$ be an abelian $m$-dimensional $sL_\infty$ algebra concentrated in degree $n$.
	This is a model of the Eilenberg-Mac Lane space $K\left(\Q^{\times m},n\right)$.
	If $S_m$ acts by  permuting the generators $u_i$, then $L^G$ is the $1$-dimensional linear span of $u_1+\cdots + u_m$.
	Thus, $\MC\left(L\right)^{hG} \simeq K\left(\Q,m\right)$. 
\end{example}

The example above is rather trivial, but we wanted to raise the following point: 
the homotopy groups of the homotopy fixed points of a connected $G$-space $\MC(L)$ are dominated by the homotopy groups of $\MC(L)$.
Since the homotopy groups of a product of Eilenberg-Mac Lane spaces are totally controllable, it is very easy to compute homotopy fixed points.
On the other extreme, the homotopy groups of a wedge of spaces are typically difficult to compute.

\begin{example}{\bf{(Wedge of 2-spheres)}}	
Consider the simplicial model of the wedge $S^2 \vee S^2$ with $u_1,u_2$ as non-degenerate $2$-simplices. 
Endow the wedge with the $\Z_2$-action that permutes the simplices $u_1$ and $u_2$.
This is probably the simplest non-trivial example of a  $G$-space of the form $X^{\vee n}$ with $S_n$-action permuting the wedge factors. 
The homotopy groups of the homotopy fixed points of such a $G$-space are very complicated to express, compared to the actual fixed points. 
Indeed, 
let $L=\mathbb{L}(u_1,u_2)$ be the free graded Lie algebra with $|u_1|=|u_2|=1$, and endow it with the trivial differential.
This is a Quillen model of $S^2 \vee S^2$ in which the base point is represented by the $0$ Maurer-Cartan element.
The $\Z_2$-action on the model, interchanging $u_1$ and $u_2$, allows us in principle to compute the homotopy groups of $\left(S^2\vee S^2\right)^{h\Z_2}$. 
Since the differential in $L$ is trivial, these homotopy groups are given by $$\pi_*\left(\MC\left(L\right)^{h{\Z_2}}\right) = \pi_*\left(\MC\left(L^{\Z_2}\right)\right) = H_{*-1}\left(L^{\Z_2}\right) = H_{*-1}(L)^{\Z_2} = s^{-1}\left(L^{\Z_2}\right).$$
The shift in degree above arises because we are using a Quillen model, which is a non-shifted dg Lie algebra.
There are many $\Z_2$-fixed elements in $L$.
In fact, there are non-trivial fixed point subspaces of arbitrarily high dimension. 
The table below shows the computation of $\mathbb{Z}_2$-invariants in degrees $\leq 3$, were we can already appreciate that there is a rich and interesting space of fixed points.

\begin{center}

\begin{table}[h!]\centering	\resizebox{.8\textwidth}{!}{
\begin{tabular}{c|c|c|c}
	degree & linear basis of $L_*$              & image of basis under $\sigma$     & linear basis of $L^{\mathbb{Z}_2}_*$ \\ \hline
	1 & $u_1, u_2$ & $u_2,u_1$ & $u_1 + u_2$ \\
	2      & $[u_1,u_2], [u_1,u_1],[u_2,u_2]$   & $[u_1,u_2], [u_2,u_2],[u_1,u_1]$  & $[u_1,u_2], [u_1,u_1]+[u_2,u_2]$     \\
	3      & $[u_1,[u_1,u_2]], [u_2,[u_1,u_2]]$ & $[u_2,[u_1,u_2]],[u_1,[u_1,u_2]]$ & $[u_1,[u_1,u_2]] + [u_2,[u_1,u_2]]$ 
\end{tabular}
		}
\end{table}

\end{center}

\end{example}

In the following example,
we explicitly compute up to a certain degree the complete $sL_\infty$ algebra corresponding to a $\mathbb{Z}_2$-space whose fundamental group is non-trivial.

\begin{example} {\bf{(The homotopy fixed points of $\Map\left(S^2, S^3\vee S^3\right)$)}}
	Endow $S^2$ with the antipodal $\mathbb{Z}_2$-action.
	Given that $S^2$ is formal, we can take its CDGA model to be the cohomology algebra $A = H^*\left(S^2;\Q\right) = \Q 1 \oplus \Q x$, where $|x|=2$, and equip it  with the $\mathbb{Z}_2$-action given by $\sigma x = -x$.
	On the other hand, consider $S^3 \vee S^3$ as a simplicial set with two non-degenerate $3$-simplices $a,b$.
	Endow it with the $\mathbb{Z}_2$-action that permutes $a$ and $b$.
	Take the Quillen model of $S^3\vee S^3$, namely the dg Lie algebra 
	$L = \left(\mathbb{L}(a,b),0\right)$ with $|a|=|b|=2$ and trivial differential. 
	Endow it with the induced $\mathbb{Z}_2$-action that interchanges $a$ and $b$.
	This is a $\mathbb{Z}_2$-DGL model of $S^3 \vee S^3$.
	Below, we find a table describing the space of $\mathbb{Z}_2$-invariants up to degree $7$. 
\end{example}

\begin{center}
	
	\begin{table}[h!]\centering	\resizebox{.9\textwidth}{!}{
\begin{tabular}{c|c|c|c}
	degree &
	linear basis of $A \otimes L$ &
	image of basis under $\sigma$ &
	linear basis of $(A\otimes L)^{\mathbb{Z}_2}_*$ \\ \hline
	1 & $x\otimes a$                             & $-x\otimes b$                            & $x\otimes(a-b)$                                       \\
	& $x\otimes b$                             & $-x\otimes a$                            &                                                       \\ \hline
	2 & $\cdot$                                  & $\cdot$                                  & $\cdot$                                               \\ \hline
	3 & $1\otimes a, 1\otimes b,$                & $1\otimes b, 1\otimes a,$                & $1\otimes(a+b), x\otimes [a,b]$                       \\
	& $x\otimes [a,b]$                         & $x\otimes [a,b]$                         &                                                       \\ \hline
	4 & $\cdot$                                  & $\cdot$                                  & $\cdot$                                               \\ \hline
	5 & $1\otimes [a,b], $                       & $-1\otimes[a,b],$                        & $x\otimes\left([a,[a,b]] + [b,[a,b]]\right)$          \\
	& $x\otimes [a,[a,b]], x\otimes [b,[a,b]]$ & $x\otimes [b,[a,b]], x\otimes [a,[a,b]]$ &                                                       \\ \hline
	6 & $\cdot$                                  & $\cdot$                                  & $\cdot$                                               \\ \hline
	7 &
	$1\otimes [a,[a,b]], 1\otimes [b,[a,b]],$ &
	$-1\otimes [b,[a,b]], -1\otimes [a,[a,b]]$ &
	$1\otimes\left([a,[a,b]]-[b,[a,b]]\right)$ \\
	&
	$x\otimes [a,[a,[a,b]]], x\otimes [b,[a,[a,b]]], x\otimes [b,[b,[a,b]]]$ &
	$x\otimes [b,[b,[a,b]]], x\otimes [a,[b,[a,b]]], x\otimes [a,[a,[a,b]]]$ &
	$x\otimes [b,[a,[a,b]]]$ \\
	&                                          &                                          & $x\otimes \left([a,[a,[a,b]]] + [b,[b,[a,b]]]\right)$
\end{tabular}
		}
	\end{table}

\end{center}

\section{Comparison to other Maurer-Cartan spaces} \label{Sec: Comparison}

There are several solutions to the problem of integrating $L_\infty$ algebras, that is, to the geometric or simplicial realization of $L_\infty$ algebras. 
In this paper, we focus on the Deligne-Hinich $\infty$-groupoid $\MC(L)$. 
Other alternative constructions include Getzler's $\gamma_\bullet(L)$ \cite{Get09},
Buijs--Félix--Murillo--Tanré's $\langle L\rangle_\bullet$ \cite{Bui20} (for complete dg Lie algebras),
 and Robert-Nicoud--Vallette's $R_\bullet(L)$ \cite{Rob20} 
 (which is the natural extension of $\langle L\rangle_\bullet$ to complete $L_\infty$-algebras). 
Each of these constructions have advantages and disadvantages compared to each other, 
none of which concerns us here.
It is known, see the references just given, 
that all the just mentioned simplicial sets are homotopy equivalent Kan complexes whenever they are all defined. 
In this section, we prove that our statements hold for any of these alternative constructions. 
The main result of this section is Corollary \ref{Corollary: comparacion}.

\medskip

Since we have been working with $sL_\infty$ algebras up to this point, we start comparing our results to the work done in \cite{Rob20}.
In loc. cit., an adjoint pair of functors is introduced:
$$\mathcal L : \mathsf{sSet} \leftrightarrows \mathsf{s \widehat L_\infty^{St}} : R.$$
Recall that $\mathsf{s \widehat L_\infty^{St}} $ is the category of complete shifted $L_\infty$ algebras equipped with \emph{strict} morphisms.
The first step is to extend the adjoint pair above to an adjunction between the corresponding categories of $G$-objects.

\begin{proposition} \label{Prop:Adjoint pair}
	For any group $G$, the pair $\left(\mathcal L,R\right)$ above extends to an adjoint pair of functors 
	$$\mathcal L: \GsSet \leftrightarrows G\textrm{-}\mathsf{s\widehat L_\infty^{St}} : R.$$
\end{proposition}

\begin{proof}
Let us define the functors on objects and morphisms:
\begin{enumerate}
	\item Let $X_\bullet$ be a $G$-simplicial set. 
	Recall that
	$\mathcal L\left(X_\bullet\right)$ is quasi-free on the normalized simplicial chains $C_*\left(X_\bullet;\Q\right)$. That is, 
	the generators of  $\mathcal L(X_\bullet)$ as a free $sL_\infty$ algebra are in bijection with the non-degenerate simplices of $X_\bullet$. 
	The $G$-action permutes these generators as dictated by the original $G$-action on $X_\bullet$, and it extends to the $sL_\infty$-structure by
	\begin{equation*}
		g \cdot \ell_n\left(x_{1},...,x_n\right) = \ell_n\left(g\cdot x_1 ,...,g\cdot x_n\right).
	\end{equation*} With this definition, given a $G$-equivariant map $f:X_\bullet \to Y_\bullet$, the induced map $\mathcal L\left(f\right)$ becomes $G$-equivariant.

\item Let $L$ be a complete $G$-$sL_\infty$-algebra. Recall that 
\begin{equation}\label{Ecu:Realization}
	R(L) = \mathsf{s \widehat L_\infty^{St}}\left(mc^\bullet, L\right),
\end{equation} where $mc^\bullet$ is the universal Maurer-Cartan algebra (see \cite[Def. 2.8]{Rob20}), a cosimplicial complete $sL_\infty$-algebra. Each $mc^n$ is therefore a certain complete $sL_\infty$ algebra on the standard $n$-simplex $\Delta^n$. Equip each $mc^n$ with the trivial $G$-action, and endow each Hom-space of Equation (\ref{Ecu:Realization}) dimension-wise with the conjugation action: if $x\in R_n(L) = \mathsf{s \widehat L_\infty^{St}}\left(mc^n, L\right)$ and $g\in G,$ then $g\cdot x:mc^n\to L$ is the $n$-simplex of $R(L)$ defined by $$z\mapsto (g\cdot x)(z) := g\cdot x(z).$$ Then $R(L)$ is a $G$-simplicial set, and a $G$-equivariant map of $sL_\infty$ algebras produces a $G$-equivariant map between the corresponding $G$-simplicial sets.
\end{enumerate}
The adjunction between these $G$-objects is given by the restriction of the adjunction between the original objects.
\end{proof}

One can also extend the $G$-action on a complete $sL_\infty$ algebra $L$ to a $G$-action on the simplicial set  $\gamma_\bullet(L)$ by using a known isomorphism $\gamma_\bullet(L)\cong R(L)$ \cite[Thm. 2.17]{Rob20},
and extend all our results from $R(L)$ to $\gamma_\bullet(L)$ using this isomorphism.
On the other hand, 
for complete dg Lie algebras, 
the adjoint pair of Buijs--Félix--Murillo--Tanré
also extends to an adjunction 
$$\mathfrak L : \GsSet \leftrightarrows G\textrm{-}\mathsf{cDGL} : \langle - \rangle.$$
In fact, by mimicking the arguments we give for the pair $(\mathcal L,R)$, 
one obtains the exact analogous result for $(\mathfrak{L}, \langle - \rangle)$. 
Thus, from now on, we  omit the details on the alternative constructions $\gamma_\bullet(L)$ and $\langle L \rangle$ and focus just on the relationship between $\MC(L)$ and $R(L)$.

\begin{proposition} \label{Prop: Equivalent homotopy fixed points}
	For any group $G$ and complete $G$-$sL_\infty$ algebra $L$, there are an isomorphism of simplicial sets, and a homotopy equivalence
	\begin{equation*}
\gamma_\bullet(L) \xrightarrow{\cong } 	R(L) \xrightarrow{\simeq} \MC(L),
\end{equation*} both of which are $G$-equivariant.
In particular, there are induced isomorphism and homotopy equivalences when taking homotopy fixed points,
$$\gamma_\bullet(L)^{hG} \cong  R(L)^{hG} \simeq \MC(L)^{hG}.$$
\end{proposition}

\begin{proof}
	Disregarding the $G$-action, this result is known, see for instance \cite[Thm. 2.17]{Rob20}. 
	Next, we  check that the known equivalences are compatible with the $G$-action.
	As seen in the mentioned reference, there is a sequence of maps 
	\begin{equation*}
		R(L) \cong MC\left(L\otimes C_\bullet\right) \xrightarrow{I_\bullet} MC\left(L\otimes \Omega_\bullet\right) = \MC(L).
	\end{equation*} 
The canonical identification $R(L) \cong MC\left(L\otimes C_\bullet\right)$
transports the $G$-action of $R(L)$ into a $G$-action on $MC\left(L\otimes C_\bullet\right)$ where $G$ acts trivially on $C_\bullet,$ and as  dictated by the original action on $L$.
The map $I_\bullet$ is built as follows. First, the Dupont contraction from $\Omega_\bullet$ to $C_\bullet$ can be shown to endow $L\otimes C_\bullet$ with a canonical $sL_\infty$ structure together with an $sL_\infty$ morphism $I:L\otimes C_\bullet \to L\otimes \Omega_\bullet$. The components of this $sL_\infty$ morphism, which when restricted to Maurer-Cartan elements constitutes $I_\bullet$, is made up from sums and compositions of the tensor product of the original injection of elementary Whitney forms $i: C_\bullet \to \Omega_\bullet$  with the identity on $L$ (see for instance \cite[Sec. 3.2]{Rob19}). Given that $G$ acts trivially on $C_\bullet$ and $\Omega_\bullet$,  and $G$ acts as it originally did on $L$, the resulting map $I_\bullet$ is $G$-equivariant.

To finish, recall that a homotopy equivalence that is $G$-equivariant induces a homotopy equivalence on homotopy fixed points (Proposition \ref{Prop: w.e. + G map implies w.e. homotopy fixed points}).
\end{proof}

The following is the main result of this section. 
The analogous statement for the functors $\gamma_\bullet(L)$ and $\langle L\rangle$ also hold, but are omitted from the statement for conciseness.

\begin{corollary}\label{Corollary: comparacion}
	For any group $G$ and complete $G$-$sL_\infty$ algebra $L$, there is a commutative square whose vertical arrows are weak equivalences
	\begin{center}
\begin{tikzcd}
	R(L)^G \arrow[r, hook] \arrow[d, "\simeq"] & R(L)^{hG} \arrow[d, "\simeq"] \\
	\MC(L)^G \arrow[r, hook]                   & \MC(L)^{hG}                  
\end{tikzcd}
	\end{center} 
If $G$ is finite and $L=L_{\geq 1}$ is a positively graded complete $sL_\infty$ algebra, then all arrows are homotopy equivalences. 
In particular, the natural inclusion $R(L)^G \hookrightarrow R(L)^{hG}$ is a homotopy equivalence of Kan complexes.
\end{corollary}

\begin{proof}
	The commutativity of the square follows from its construction. 
	Let us show the vertical arrows are homotopy equivalences: 
\begin{itemize}
	\item The vertical left arrow. Since for any complete $sL_\infty$ algebra $M$ there is a homotopy equivalence $R(M)\xrightarrow{\simeq} \MC(M)$, and the functors $R$ and $\MC$ commute on the nose with taking $G$-fixed points, (see Remark \ref{Remark1}, the same proof works for $R(L)$) we have $$R(L)^G = R\left(L^G\right) \xrightarrow{\simeq } \MC\left(L^G\right) = \MC(L)^G.$$
	\item The vertical right arrow is a homotopy equivalence by Proposition \ref{Prop: Equivalent homotopy fixed points}.
\end{itemize}
Assuming the extra hypotheses, we easily see that the rest of arrows are homotopy equivalences: the bottom arrow is a homotopy equivalence by the main result, Theorem \ref{Teo:Main}. Finally, it follows from the 2-out-3 property that the arrow $R(L)^G \hookrightarrow R(L)^{hG}$ is a homotopy equivalence.
\end{proof}

The following result will not be explicitly used, but it is worth recording as a basic fact of the homotopy theory of complete $G$-$s\widehat L_\infty$ algebras.
The analogous statement holds true in the category $G$-$\mathsf{cDGL}$,
and the proof in that case is essentially the same.

\begin{theorem}
	For a finite group $G$, there is a cofibrantly generated model category structure on $G\textrm{-}\mathsf{s\widehat L_\infty^{St}}$ such that:
	\begin{enumerate}
		\item the weak equivalences are the $G$-equivariant strict $sL_\infty$ morphisms $f:L\to L'$ 
		whose underlying map is a weak equivalence when forgetting the $G$-action,
		\item the fibrations are the $G$-equivariant strict $sL_\infty$ morphisms $f:L\to L'$ 
		whose underlying map is fibration when forgetting the $G$-action, and
		\item the cofibrations are determined by the lifting property against acyclic fibrations.
	\end{enumerate}
	The canonical maps $\mathcal L \left(\partial \Delta^n \times G\right) \hookrightarrow \mathcal{L} \left(\Delta^n\times G\right)$, for $n\geq 0$, form a set of generating cofibrations, 
	and the canonical maps $\mathcal L\left(\Delta_k^n \times G \right)\hookrightarrow \mathcal L\left(\Delta^n \times G\right)$, for $n\geq 0$, form a set of generating acyclic cofibrations. 
\end{theorem}

\begin{proof}
	Consider the adjoint pair of functors $$\mathcal L : \GsSet \leftrightarrows G\textrm{-}\mathsf{s\widehat L_\infty^{St}} : R$$
	described at the beginning of this section.
	Since the category $G\textrm{-}\mathsf{s\widehat L_\infty^{St}}$ has finite limits and small colimits, 
	the transfer principle \cite[Thm. 11.3.2]{Hir09} applies. 
	Remark that we can also transfer the same model structure directly from $\mathsf{sSet}$ by composing the corresponding two left adjoints, and the two right adjoints.
\end{proof}

\bibliographystyle{plain}
\bibliography{MyBib}

\begin{thebibliography}{10}

\bibitem{All93}
C.~Allday and V.~Puppe.
\newblock {\em Cohomological methods in transformation groups}, volume~32 of
  {\em Cambridge Studies in Advanced Mathematics}.
\newblock Cambridge University Press, Cambridge, 1993.

\bibitem{Ber03}
I.~Berger, C.and~Moerdijk.
\newblock Axiomatic homotopy theory for operads.
\newblock {\em Comment. Math. Helv.}, 78(4):805--831, 2003.

\bibitem{Ber15}
A.~Berglund.
\newblock Rational homotopy theory of mapping spaces via {L}ie theory for
  ${L}_\infty$-algebras.
\newblock {\em Homology, Homotopy Appl.}, 17(2):343--369, 2015.

\bibitem{Bou76}
A.~K. Bousfield and V.~K. A.~M. Gugenheim.
\newblock On {${\rm PL}$} de {R}ham theory and rational homotopy type.
\newblock {\em Mem. Amer. Math. Soc.}, 8(179):ix+94, 1976.

\bibitem{Bou72}
A.~K. Bousfield and D.~M. Kan.
\newblock {\em Homotopy limits, completions and localizations}.
\newblock Lecture Notes in Mathematics, Vol. 304. Springer-Verlag, Berlin-New
  York, 1972.

\bibitem{Bre72}
G.~E. Bredon.
\newblock {\em Introduction to compact transformation groups}.
\newblock Academic Press, New York-London, 1972.
\newblock Pure and Applied Mathematics, Vol. 46.

\bibitem{Bui15}
U.~Buijs, Y.~F{\'e}lix, S.~Huerta, and A.~Murillo.
\newblock The homotopy fixed point set of {L}ie group actions on elliptic
  spaces.
\newblock {\em P. Lond. Math. Soc.}, 110(5):1135--1156, 2015.

\bibitem{Bui09}
U.~Buijs, Y.~F{\'e}lix, and A.~Murillo.
\newblock Rational homotopy of the (homotopy) fixed point sets of circle
  actions.
\newblock {\em Adv. in Math.}, 222(1):151--171, 2009.

\bibitem{Bui20}
U.~Buijs, Y.~F\'{e}lix, A.~Murillo, and D.~Tanr\'{e}.
\newblock {\em Lie models in topology}, volume 335 of {\em Progress in
  Mathematics}.
\newblock Birkh\"{a}user/Springer, Cham, 2020.

\bibitem{Car91}
G.~Carlsson.
\newblock Equivariant stable homotopy and {S}ullivan's conjecture.
\newblock {\em Invent. Math.}, 103(3):497--525, 1991.

\bibitem{Car92}
G.~Carlsson.
\newblock A survey of equivariant stable homotopy theory.
\newblock {\em Topology}, 31(1):1--27, 1992.

\bibitem{Cur71}
E.~B. Curtis.
\newblock Simplicial homotopy theory.
\newblock {\em Advances in Math.}, 6:107--209, 1971.

\bibitem{Dol15}
V.~A. Dolgushev and C.~L. Rogers.
\newblock A version of the {G}oldman--{M}illson theorem for filtered
  {$L_\infty$}-algebras.
\newblock {\em Journal of Algebra}, 430:260--302, 2015.

\bibitem{Yve12}
Y.~F{\'e}lix, S.~Halperin, and J-C Thomas.
\newblock {\em Rational homotopy theory}, volume 205.
\newblock Springer Science \& Business Media, 2012.

\bibitem{Goe99}
Goerss~P. G. and Jardine~J. F.
\newblock {\em {Simplicial homotopy theory.}}, volume 174.
\newblock Basel: Birkh\"auser, 1999.

\bibitem{Get09}
E.~Getzler.
\newblock Lie theory for nilpotent $l_\infty$-algebras.
\newblock {\em Ann. of Math.}, pages 271--301, 2009.

\bibitem{Gol98}
M.~Golasi\'{n}ski.
\newblock Equivariant rational homotopy theory as a closed model category.
\newblock {\em J. Pure Appl. Algebra}, 133(3):271--287, 1998.

\bibitem{Goy89}
J.~O. Goyo.
\newblock {\em The {S}ullivan model of the homotopy-fixed-point set}.
\newblock ProQuest LLC, Ann Arbor, MI, 1989.
\newblock Thesis (Ph.D.)--University of Toronto (Canada).

\bibitem{Hin97A}
V.~Hinich.
\newblock Descent of {D}eligne groupoids.
\newblock {\em Internat. Math. Res. Notices}, (5):223--239, 1997.

\bibitem{Hir09}
P~S Hirschhorn.
\newblock {\em Model categories and their localizations}.
\newblock Number~99. American Mathematical Soc., 2009.

\bibitem{Kon03}
M.~Kontsevich.
\newblock Deformation quantization of {P}oisson manifolds.
\newblock {\em Lett Math Phys.}, 66(3):157--216, 2003.

\bibitem{May96}
J.~P. May.
\newblock {\em Equivariant homotopy and cohomology theory}, volume~91 of {\em
  CBMS Regional Conference Series in Mathematics}.
\newblock Published for the Conference Board of the Mathematical Sciences,
  Washington, DC; by the American Mathematical Society, Providence, RI, 1996.

\bibitem{Mil84}
H.~Miller.
\newblock The {S}ullivan conjecture on maps from classifying spaces.
\newblock {\em Ann. of Math}, 120:39--87, 1984.

\bibitem{Prid10}
J.~P. Pridham.
\newblock Unifying derived deformation theories.
\newblock {\em Adv. Math.}, 224(3):772--826, 2010.

\bibitem{Rob19}
D.~Robert-Nicoud.
\newblock Representing the deformation $\infty$-groupoid.
\newblock {\em Algebr Geom Topol.}, 19(3):1453--1476, 2019.

\bibitem{Rob20}
D.~Robert-Nicoud and B.~Vallette.
\newblock Higher {L}ie theory.
\newblock {\em arXiv preprint 2010.10485}, 2020.

\bibitem{Rog20}
Christopher~L Rogers.
\newblock Complete filtered {$L_\infty$}-algebras and their homotopy theory.
\newblock {\em arXiv preprint 2008.01706}, 2020.

\bibitem{Scu02}
L.~Scull.
\newblock Rational {$S^1$}-equivariant homotopy theory.
\newblock {\em Trans. Amer. Math. Soc.}, 354(1):1--45, 2002.

\bibitem{Scu08}
L.~Scull.
\newblock A model category structure for equivariant algebraic models.
\newblock {\em Trans. Amer. Math. Soc.}, 360(5):2505--2525, 2008.

\bibitem{Sul77}
D.~Sullivan.
\newblock Infinitesimal computations in topology.
\newblock {\em Inst. Hautes \'{E}tudes Sci. Publ. Math.}, (47):269--331, 1977.

\bibitem{Tri82}
G.~V. Triantafillou.
\newblock Equivariant minimal models.
\newblock {\em Trans. Amer. Math. Soc.}, 274(2):509--532, 1982.

\bibitem{Bar93}
B.~Zwiebach.
\newblock Closed string field theory: quantum action and the
  {B}atalin-{V}ilkovisky master equation.
\newblock {\em Nuclear Phys. B}, 390(1):33--152, 1993.

\end{thebibliography}

\bigskip

\noindent\sc{José Manuel Moreno Fernández}\\ 
\noindent\sc{Departamento de Álgebra, Geometría y Topología, \\ Universidad de Málaga, 29080 \\ Málaga, Spain}\\
\noindent\tt{josemoreno@uma.es}\\

\medskip

\noindent\sc{Felix Wierstra}\\ 
\noindent\sc{Korteweg-de Vries Institute for Mathematics \\ University of Amsterdam \\ Science Park 105-107, 1098 XG Amsterdam, Netherlands}\\
\noindent\tt{felix.wierstra@gmail.com}

\end{document}